\documentclass[a4paper,10pt,reqno]{amsart} %
\usepackage{latexsym,amssymb}
\usepackage{hyperref}
\usepackage{amsmath,amsthm,amsxtra}
\usepackage{amsfonts}
\usepackage[american]{babel}
\usepackage{mathrsfs}
\usepackage{psfrag}
\usepackage{epsfig,inputenc}
\usepackage{graphpap,latexsym,epsf}
\usepackage{color}
\usepackage{amssymb,eucal,paralist,color,enumerate}
\usepackage{graphicx}

\newcommand{\R}{\mathbb{R}}
\newcommand{\N}{\mathbb{N}}

\mathchardef\emptyset="001F

\newtheorem{maintheorem}{Theorem}

\newtheorem{theorem}{Theorem}[section]

\newtheorem{lemma}[theorem]{Lemma}

\newtheorem{definition}[theorem]{Definition}
\newtheorem{proposition}[theorem]{Proposition}

\numberwithin{equation}{section}

\newcommand{\up}{\uparrow}

\newcommand{\down}{\downarrow}
\newcommand{\weaksto}{\rightharpoonup^*}

\newcommand{\QED}{\mbox{}\hfill\rule{5pt}{5pt}\medskip\par}

 \def\calE{{\mathcal E}} \def\calF{{\mathcal F}}
  \def\calI{{\mathcal I}}

  \def\calR{{\mathcal R}}
 \def\calT{{\mathcal T}} 
  
\def\calY{{\mathcal Y}} 

  \def\rmC{{\mathrm C}}
\def\rmD{{\mathrm D}}

\def\dd{\;\!\mathrm{d}} 

\newcommand{\eps}{\varepsilon}
\newcommand{\teta}{\vartheta}
\newcommand{\foraa}{\text{for a.a.\ }}

\newcommand{\Xs}{X}
 \newcommand{\AC}{\mathrm{AC}}
   \newcommand{\BV}{\mathrm{BV}}
\newcommand{\mdn}{\mathsf{d}}
\newcommand{\corn}{\delta}
\newcommand{\md}[2]{\mathsf{d}(#1,#2)}
\newcommand{\ene}[2]{\mathcal{E}(#1,#2)}
\newcommand{\pertn}{\mathcal{F}}
\newcommand{\pert}[2]{\mathcal{F}(#1,#2)}
\newcommand{\perto}[1]{\mathcal{F}_0(#1)}
\newcommand{\pwn}{\mathcal{P}}
\newcommand{\pw}[2]{\mathcal{P}(#1,#2)}
\newcommand{\ds}[3]{{#1}_{#2}^{#3}}
\newcommand{\cmdn}{\mathsf{D}}
\newcommand{\cmd}[2]{\mathsf{D}(#1,#2)}
\newcommand{\corr}[2]{\delta(#1,#2)}
\newcommand{\corrn}[2]{\delta_n(#1,#2)}
\newcommand{\Vars}[3]{\mathrm{Var}_{#1}(#2,#3)}
\newcommand{\Vari}[4]{\mathrm{Var}_{#1}(#2,[#3,#4])}

\newcommand{\Varname}[1]{\mathrm{Var}_{#1}}
\newcommand{\Vf}[2]{V_{#1}(#2)}
\newcommand{\lli}[2]{{#1}({#2}{-})}
\newcommand{\rli}[2]{{#1}({#2}{+})}
\newcommand{\jump}[1]{\mathrm{J}_{#1}}
\newcommand{\tjump}{\widetilde{\mathrm{J}}}
\newcommand{\Jvar}[4]{\mathrm{Jmp}_{#1}(#2;[#3,#4])}
\newcommand{\mum}[2]{\nu_{#1}^{#2}}
\newcommand{\gcostname}{\mathsf{e}}
\newcommand{\gcost}[3]{\mathsf{e}(#1,#2,#3)}
\newcommand{\vecostname}{\mathsf{c}}
\newcommand{\vecost}[3]{\mathsf{c}(#1,#2,#3)}
\newcommand{\vecostnamep}[1]{\mathsf{c}_{#1}}

\newcommand{\bvcostname}{\mathsf{v}}
\newcommand{\bvcost}[3]{\mathsf{v}(#1,#2,#3)}
\newcommand{\idelta}[4]{\Delta_{#1}(#2,#3,#4)}
\newcommand{\aidelta}[5]{\Delta_{#1}(#2,#3,#4,#5)}
\newcommand{\RIS}{(\Xs,\calE,\mdn)}
\newcommand{\VE}{\mathrm{VE}}
\newcommand{\VEa}[1]{\mathrm{VE}_{#1}}
\newcommand{\VEn}{\VEa{\mu_n}}
\newcommand{\stab}[1]{\mathscr{S}_{#1}}
\newcommand{\stabi}[2]{\mathscr{S}_{#1}(#2)}
\newcommand{\rstab}[2]{\mathcal{R}(#1,#2)}
\newcommand{\slope}[3]{|\rmD {#1}|(#2,#3)}
\newcommand{\hole}[1]{\mathfrak{h}(#1)}
\newcommand{\tcost}[4]{\mathrm{Trc}_{#1}(#2,#3,#4)}
\newcommand{\tcostname}[1]{\mathrm{Trc}_{#1}}
\newcommand{\Gap}[3]{\mathrm{GapVar}_{#1}(#2,#3)}
\newcommand{\limE}{\mathsf{E}}
\newcommand{\limV}{\mathsf{V}}
\newcommand{\limP}{\mathsf{P}}
\newcommand{\subl}[1]{S_{#1}}
\newcommand{\nresc}[1]{\mathfrak{s}_{#1}}

\newcommand{\ninvresc}[1]{\mathfrak{t}_{#1}}

\newcommand{\invcur}[1]{\mathfrak{u}_{#1}}
\newcommand{\invcu}{\mathfrak{u}}
\newcommand{\newmresc}[1]{\mathfrak{m}_{#1}}

\newcommand{\rresc}[2]{\mathsf{\sigma}_{#1}^{#2}}
\newcommand{\siresc}[2]{\mathsf{\tau}_{#1}^{#2}}
\newcommand{\piecewiseConstant}[2]{\overline{#1}_{\kern-1pt#2}}
\newcommand{\pwc}{\piecewiseConstant}
\newcommand{\kliminf}{\mathrm{Li}}
\newcommand{\klimsup}{\mathrm{Ls}}

\definecolor{dmagenta}{rgb}{0.8,0,0.8}

\newcommand{\RRR}{\color{red}}

\newcommand{\EEE}{\color{black}}

\definecolor{vgreen}{rgb}{0.1,0.5,0.2}

\title[From Visco-Energetic to Energetic and $\BV$ solutions]{From Visco-Energetic to Energetic and Balanced Viscosity solutions of rate-independent systems}

\author{Riccarda Rossi}
\address{R.\ Rossi, DIMI, Universit\`a degli studi di Brescia, via Branze 38, 25133 Brescia - Italy}
\email{riccarda.rossi\,@\,unibs.it}

\author{Giuseppe Savar\'e}
\address{G.\ Savar\'e, Dipartimento di Matematica, Universit\`a degli studi di Pavia, via Ferrata 1, 27100 Pavia - Italy}
\email{giuseppe.savare\,@\,unipv.it}

\thanks{
  R.R.\   acknowledges
  support from the  Gruppo Nazionale per  l'Analisi Matematica, la
  Probabilit\`a  e le loro Applicazioni (GNAMPA) of the Istituto Nazionale di Alta Matematica (INdAM)}

\date{April 04, 2017}
\begin{document}

\begin{abstract} This  paper focuses on  weak solvability concepts for rate-independent systems in  a metric setting.
\emph{Visco-Energetic} solutions have been recently obtained by passing to the time-continuous limit in  a time-incremental scheme, akin to that for Energetic solutions, but perturbed by a `viscous' correction term, as in the case of Balanced Viscosity solutions. However, for Visco-Energetic solutions this viscous correction  is tuned by a \emph{fixed} parameter $\mu$. The resulting solution notion 
is characterized by a stability condition and an energy balance analogous to those for Energetic solutions, but, in addition, it provides a fine description of the system behavior at jumps as Balanced Viscosity solutions do. Visco-Energetic evolution can be thus thought as `in-between' Energetic and Balanced Viscosity evolution.
\par
Here
we aim to formalize this intermediate character of Visco-Energetic solutions by studying their singular limits as $\mu\down 0$ and $\mu\uparrow \infty$. We shall prove convergence to Energetic solutions in the former case, and to Balanced Viscosity solutions in the latter situation. 
\end{abstract}

\maketitle

\centerline{\emph{Dedicated to Gianni Gilardi on the occasion of his 70th birthday}}


\section{Introduction}
A large class of  \emph{rate-independent systems} are driven by
 \begin{itemize}
 \item[-]
  a time-dependent energy functional $\calE: [0,T]\times \Xs \to (-\infty,\infty]$, with $[0,T]$ the time span during which the system is observed, and $\Xs$ the space of the states of the system, 
\item[-]
 a (positive) dissipation functional $\mathcal{D}: \Xs \times \Xs \to [0,\infty)$, keeping track of the energy dissipated by the curve $u:[0,T]\to \Xs$ describing the evolution of the system, 
  that satisfies suitable structural properties 
 peculiar of rate-independence. 
\end{itemize} 
\par
When $\Xs$ is a (separable) Banach space, 
a natural class of dissipations is provided by translation invariant functionals of the form $\mathcal{D}
(u_1,u_2): = \Psi(u_2{-}u_1)$, where 
$\Psi: \Xs \to [0,\infty)$ is a  (convex, lower semicontinuous) dissipation potential, 
  positively homogeneous of degree $1$, namely $\Psi(\lambda v) = \lambda \Psi(v)$ for all $\lambda\geq0$ and $v\in \Xs$.
The evolution of the rate-independent system 
is governed by the   doubly nonlinear differential inclusion 
\begin{equation}
\label{DNE}
\partial \Psi(u'(t)) + \partial_u \calE(t,u(t)) \ni 0 \quad \text{in } \Xs^* \quad \foraa\, t \in (0,T),
\end{equation}
where 
 $\partial\Psi: X \rightrightarrows X^*$ is the subdifferential in the sense of convex analysis, while  $\partial_u \calE : [0,T] \times X \rightrightarrows X^*$ is a suitable notion of subdifferential  of $\calE$ w.r.t.\ the variable $u$. 
As it will be apparent from the forthcoming discussion,  in general \eqref{DNE} is only formally written. 
\par
More generally, throughout
 this paper  we shall assume that 
 the dissipation $\mathcal{D}$ is induced by a distance $\mdn$ on 
 the space $\Xs$, such that
\begin{equation}
\label{met-sp}
\tag{$\mathrm{X}$}
\text{$(\Xs,\mdn)$ is a \emph{complete} metric space}.
\end{equation}
We will henceforth denote a (metric) rate-independent system by $\RIS$. 
%
\par
Rate-independent evolution occurs in manifold  problems in physics and engineering, cf.\ \cite{Miel05ERIS,MieRouBOOK} for a survey. In addition to its wide range of applicability, over the last two decades the analysis of rate-independent systems has attracted considerable interest due to its intrinsic mathematical challenges: first and foremost, the quest of a proper solvability concept for the system $\RIS$. In fact, since the dissipation potential has linear growth at infinity, one can in general expect only $\BV$-time regularity for the curve $u$ (unless the energy functional is uniformly convex). Thus $u$ may  have jumps as a function of time. Therefore, the pointwise derivative $u'$ in the subdifferential inclusion \eqref{DNE} in the  Banach setting, and the metric  derivative $|u'|$ in the general metric setup \eqref{met-sp}, need not be defined. This calls for  a suitable weak formulation of rate-independent evolution, also able to satisfactorily capture the behavior of the system in the jump regime. 
\par
In what follows we illustrate the three solution concepts this paper is concerned with, referring to Sections \ref{s:2} and \ref{s:main-results} for more details and precise statements.
\subsection{Energetic, Balanced Viscosity, and Visco-Energetic solutions}
The pioneering papers \cite{MieThe99MMRI,MieThe04RIHM}  advanced the by now classical concept of (Global) \underline{\emph{Energetic}} solution to the rate-independent system $\RIS$ (cf.\ also  the notion of `quasistatic
evolution' in the realm of crack propagation, dating back to \cite{DM-Toa2002}), which can be in fact given in a more general topological setting \cite{MaiMie05EREM}. 
 It is a curve $u: [0,T]\to \Xs$ with bounded variation, complying
 for every $t\in [0,T]$  with 
 \begin{itemize}
 	\item[-] the global stability condition
 	\begin{equation}
 	\label{global-stab-intro}
 	\tag{$\mathrm{S}_{\mdn}$}
 	\ene t{u(t)} \leq \ene tv + \md{u(t)}{v} \quad \text{for every } v \in \Xs,
 	\end{equation}
 	\item[-]  the energy balance
 	\begin{equation}
 	\label{global-enbal-intro}
 	\tag{$\mathrm{E}_{\mdn}$}
 	\ene t{u(t)} + \Vari {\mdn}{u}0{t} = \ene 0{u(0)} +\int_0^t \partial_t \ene s{u(s)} \dd s \,.
 	\end{equation}
 \end{itemize}
Here, $ \Vari {\mdn}{u}0{t} $ denotes the (pointwise) total
variation of the curve $u$ induced by the metric $\mdn$, which is related to `energy dissipation': in fact, 
\eqref{global-enbal-intro}  balances the stored energy at the process time $t$ and the energy dissipated up to   $t$
with the initial energy and the work of the external loadings, encoded in the second integral on the right-hand side. 
 Existence results for Energetic solutions may be proved by resorting to a well understood time discretization procedure. Indeed, for every fixed partition $\calT_\tau: = \{ \ds t \tau 0=0< 
\ds t  \tau1 <\ldots<\ds t \tau{N-1}<\ds t \tau N=T\}$ of the interval 
$[0,T]$, with fineness $\tau: = \max_{i=1,\ldots,N} (\ds t \tau i {-} \ds t \tau{i-1})$, 
discrete solutions $(\ds Un\tau)_{n=1}^{N}$  are constructed as solutions of the time-incremental minimization scheme  
\begin{equation}
\label{tims-ris}
\tag{$\mathrm{IM}_\tau$}
\min_{U\in \Xs} \left( \ene{\ds t  \tau n}{U}+ \md{\ds U\tau{n-1}}{U} \right).
\end{equation}
Under suitable conditions it can be shown that, for every null sequence $(\tau_k)_k$, up to a subsequence the
  piecewise constant interpolants $(\pwc U{\tau_k})_k$ of the discrete solutions converge to  an Energetic solution.
  While widely applied,  the Energetic concept has also been criticized on the grounds that the global stability condition 
  \eqref{global-stab-intro} is too strong a requirement, when dealing with nonconvex energies. To avoid violating it, the system may in fact have to 
change instantaneously in a very drastic way, jumping into very far-apart energetic
  configurations, possibly `too early'. In this connection, we refer to the discussions from  \cite[Ex.\ 6.3]{KnMiZa07?ILMC}, \cite[Ex.\ 6.1]{MRS09}, as well as to  \cite{RosSav12}, providing  a characterization of Energetic solutions to one-dimensional 
  rate-independent systems (i.e., with $\Xs=\R$), driven by a fairly broad class of nonconvex  energies. 
  In \cite{RosSav12}, the input-output relation associated with the Energetic concept is shown to be related to the so-called \emph{Maxwell rule} for hysteresis processes \cite{Visintin94}. 
  These features are also reflected in the jump conditions satisfied by an Energetic solution $u$ at every jump point $t
  \in \jump u$ ($\lli u t$, $\rli u t $ denoting the left/right limits of $u$ at $t$ and  $\jump u$ its jump set),  namely
  \begin{equation}
  \label{jump-intro-energetic}
  \md {\lli u t }{u(t)} = \ene t{\lli u t } -\ene t{u(t)}, \qquad \md {u(t)}{\rli u t} = \ene t{u(t)} - \ene t {\rli u t },
  \end{equation}
  which show the influence of the global energy landscape of $\calE$.
\par
The global stability  condition \eqref{global-stab-intro} in fact stems from the 
  global minimization problem \eqref{tims-ris}, whereas a scheme based on \emph{local} minimization would be preferable, cf.\ \cite{DT02MQGB} for a first discussion of this  in the realm of crack propagation,
  and \cite{EfeMie06RILS} in the frame of abstract (finite-dimensional) rate-independent systems. This localization can be achieved by perturbing the variational scheme \eqref{tims-ris}
 with
a term, modulated by a viscosity parameter $\eps$, which penalizes the squared distance from the previous step $\ds U\tau{n-1}$. One is thus led to consider the  time-incremental minimization   
\begin{equation}
\label{tims-bv}
\tag{$\mathrm{IM}_{\eps,\tau}$}
\min_{U\in \Xs} \left( \ene{\ds t \tau  n}{U}+ \md{\ds U\tau{n-1}}{U} + \frac{\eps}{2\tau}   \mdn^2(\ds U\tau{n-1},U)  \right),
\end{equation}
which may be considered as  a \emph{viscous} approximation of 
\eqref{tims-ris}.
For fixed $\eps>0$, the limit passage as $\tau\down 0$ in   \eqref{tims-bv} leads to solutions (of the metric formulation) of the
 Generalized Gradient System $(\Xs,\calE,\mdn,\psi_\eps)$, where the dissipation function $\psi_\eps: [0,\infty)\to[0,\infty)$ is   given by
\begin{equation}
\label{psi-eps-intro}
\psi_\eps(r) =r+\frac{\eps}2 r^2= \frac1{\eps} \psi(\eps r) \qquad \text{with}  \quad \psi(r) = r +\frac12 r^2\,.
\end{equation}
 We refer to \cite{RMS08} for existence results for gradient systems in metric spaces, driven by dissipation potentials with superlinear  growth at infinity like $\psi_\eps$.   
     In turn, it has been shown in \cite{MRS09}  (cf.\ also \cite{MRS-MJM})
     that, under suitable conditions on the energy functional, time-continuous solutions (to the metric formulation)  of  $(\Xs,\calE,\mdn,\psi_\eps)$  converge as $\eps \down 0$,  up to reparamerization, to a \underline{\emph{Balanced Viscosity}} ($\BV$)   solution of the rate-independent system
     $\RIS$. The latter is  a curve $u\in \BV([0,T];\Xs)$
     satisfying 
      \begin{itemize}
\item[-] the \emph{local} stability condition
\begin{equation}
\label{loc-stab-INTRO}
\tag{$\mathrm{S}_{\mdn,\mathrm{loc}}$}
\slope \calE t{u(t)} \leq 1 \quad \text{for every } t \in [0,T]\setminus \jump u,
\end{equation}
\item[-]  the energy balance
\begin{equation}
\label{BV-enbal-INTRO}
\tag{$\mathrm{E}_{\mdn,\bvcostname}$}
\ene t{u(t)} + \Vari {\mdn,\bvcostname}{u}0{t} = \ene 0{u(0)} +\int_0^t  \partial_t \ene{s}{u(s)}   \dd s \quad \text{for all } t \in [0,T]\,.
\end{equation}
\end{itemize}
Here, $|\rmD \calE|: [0,T]\times \Xs \to [0,\infty]$ is the  \emph{metric slope} of the energy functional $\calE$, namely
\begin{equation}
\label{slope-INTRO}
\slope \calE tu:  = \limsup_{v\to u} \frac{(\ene tu{-}\ene tv)_+}{\md uv}\,,
\end{equation}
and $\Varname {\mdn,\bvcostname} $ is a suitably \emph{augmented} notion of total variation, fulfilling $\Vari {\mdn,\bvcostname}{u}a{b} \geq 
 \Vari {\mdn}{u}a{b} $ for all  $[a,b]\subset [0,T]$, which measures the energy dissipated along  the jump, at a point  $t\in \jump u$, by means of the cost
 \begin{equation}
 \label{vcost-INTRO}
 \begin{aligned}
 \bvcost t{\lli u t }{\rli u t}: = \inf\Big\{
   & \int_{r_0}^{r_1}
|\teta'|(r) \left( \slope \calE t{\teta(r)} \vee 1 \right)  \dd r 
  \, :
  \\
   & 
 \qquad 
   \teta \in \AC([r_0,r_1];\Xs),\, \ \teta(r_0)=\lli u t , \ \teta(r_1) = \rli u t\Big\}\,
   \end{aligned}
\end{equation}
that is reminiscent of the viscous approximation \eqref{tims-bv}. 
Indeed, it is possible to show (cf.\ \eqref{bv-jump-cond-INTR0}
ahead) that every $\BV$ solution to $\RIS$ complies with the jump conditions
\begin{equation}
\begin{aligned}
\label{bv-jump-cond-INTR0}
\ene t{\lli u t } - \ene t{\rli ut}  = \bvcost t{\lli ut}{\rli u t} =    \int_{r_0}^{r_1}
|\teta'|(r) \left( \slope \calE t{\teta(r)} \vee 1 \right)  \dd r  
\end{aligned}
 \end{equation}
at every jump point $t\in \jump u$, with $\teta$ an optimal jump transition between $\lli u t $ and $\rli u t$. Any optimal  transition can be decomposed into an (at most) countable collection of \emph{sliding} transitions, evolving in the rate-independent mode, and 
 \emph{viscous} transitions, i.e.\ (metric) solutions of the Generalized Gradient System $(\Xs,\calE,\mdn,\psi)$ with the superlinear $\psi$ from \eqref{psi-eps-intro},  and where the time variable in the energy functional is frozen at the jump time $t$. 
  Therefore, $\BV$ solutions account for 
 the onset of viscous behavior at jumps of the system, which can be in fact interpreted as fast transitions (possibly) governed by viscosity. The characterization in the one-dimensional case, with a nonconvex driving energy,  from \cite{RosSav12} reveals that the input-output relation underlying $\BV$  solutions follows the \emph{delay rule} 
 \cite{Visintin94}, as they tend to jump `as late as possible'.
 \par
 A notable feature of $\BV$ solutions is that they can be directly obtained as limits of the discrete solutions arising from the perturbed scheme \eqref{tims-bv}, when the parameters $\eps$ and $\tau$ \emph{jointly} tend to zero with convergence rates such that
 \begin{equation}
 \label{eps-tau-INTRO}
 \lim_{\eps,\tau \down 0}\frac{\eps}{\tau} = +\infty\,;
 \end{equation}
 the argument developed in  \cite{MRS12, MRS13}  in the Banach setting  can be in fact   easily   extended to the metric framework, cf.\ the discussion in Sec.\ 
 \ref{ss:2.3}.  This remarkable property has somehow inspired the approach in \cite{SavMin16}. There, a new notion of rate-independent evolution has been obtained in the time-continuous limit, as $\tau \down 0$, of the  perturbed
 time-incremental minimization scheme
\begin{equation}
\label{tims}
\tag{$\mathrm{IM}_{\mu}$}
\begin{gathered}
\min_{U\in \Xs} \left( \ene{\ds t  \tau n }{U}+ \md{\ds U\tau{n-1}}{U}  +   \frac{\mu}2 \mdn^2(\ds U\tau{n-1},U)
 \right) \quad \text{with }  \mu>0 \text{ a \emph{fixed} parameter.} 
\end{gathered}
\end{equation}
The analysis carried out in \cite{SavMin16} in fact covers a more general, topological setting,
  akin to that of \cite{MaiMie05EREM}, with a general viscous correction $ \corn: \Xs \times \Xs \to [0,\infty)$ compatible, in a suitable sense, with the metric $\mdn$:   a particular case is in fact $\corr u v =   \tfrac{\mu}2 \mdn^2(u,v)$ as in \eqref{tims}. 
  In the simplified metric setting of \eqref{met-sp}, under the same conditions ensuring the existence of Energetic solutions it is possible to show that the (piecewise constant interpolants of the) discrete solutions arising from \eqref{tims} converge, as $\tau\down 0$ and $\mu>0$ is fixed, to a   \underline{\emph{($\mu$-)Visco-Energetic}}  solution to the rate-independent system $\RIS$. 
  In what follows, we will simply speak of \emph{Visco-Energetic} ($\VE$) solutions, and often  highlight their dependence on the parameter $\mu$ in the acronym $\VEa{\mu}$. 
  A $\VEa{\mu}$ solution  is  a curve $u\in \BV([0,T];\Xs)$ complying with the 
  \begin{itemize}
  \item[-] `perturbed', still global, stability condition
  \begin{equation}
\label{stab-VE-INTRO}
\tag{$\mathrm{S}_{\cmdn}$}
\ene t{u(t)} \leq \ene tv + \md{u(t)}v + \frac\mu2 \mdn^2 (u(t),v)   \quad \text{for every } v \in \Xs \text{ and for every } t \in [0,T]\setminus \jump u,
\end{equation}
\item[-]  the energy balance
\begin{equation}
\label{enbal-VE-INTR0}
\tag{$\mathrm{E}_{\mdn,\vecostname}$}
\ene t{u(t)} + \Vari {\mdn,\vecostname}{u}0{t} = \ene 0{u(0)} +\int_0^t  \partial_t \ene s{u(s)}   \dd s \quad \text{for all } t \in [0,T]\,.
\end{equation}
\end{itemize}
Here, $\Varname {\mdn,\vecostname}$ is an alternative augmented total variation functional,
 again estimating the total variation induced by $\mdn$, but featuring  a different notion of jump dissipation cost. In analogy with \eqref{vcost-INTRO}, the visco-energetic cost  $\vecostname$  (we shall often write $\vecostnamep{\mu}$ to highlight its dependence on the parameter $\mu$, and accordingly write ($\mathrm{E}_{\mdn,\vecostnamep{\mu}}$)),  is still  obtained by minimizing a suitable transition cost 
$\tcostname{\VE}$ 
 over a class of continuous, but not necessarily absolutely continuous, curves $\teta: E \to \Xs$, with $E$ an arbitrary  compact subset of $\R$ having a possibly more complicated structure than that of an interval. 
The transition cost  $\tcostname{\VE}$  evaluates (1) the $\mdn$-total variation $\Vars{\mdn} \teta E $ of $\teta$ over $E$; (2) a quantity related to the ``gaps'' of the set $E$; (3) a quantity measuring the violation of the (global) stability condition \eqref{stab-VE} along the jump transition $\teta$, cf.\ 
 \cite{SavMin16}  and 
Sec.\ \ref{ss:2.2} ahead  for all details and precise formulae. In this context as well, it can be proved (cf.\ \cite[Prop.\ 3.8]{SavMin16}) that any $\VE$ solution $u$  satisfies at  its jump points $t\in \jump u $ the jump conditions
\begin{equation}
\begin{aligned}
\label{ve-jump-cond-INTR0}
\ene t{\lli u t } - \ene t{\rli ut}  = \vecost t{\lli ut}{\rli u t} =    \tcost{\VE} {t}{\teta}E
\end{aligned}
 \end{equation}
 with $\teta: E \to \Xs$ an optimal transition curve between $\lli u t$ and $\rli u t$. Furthermore, 
 any optimal transition can be decomposed into an (at most countable) collection of \emph{sliding transitions}, parameterized by a continuous variable and fulfilling the stability condition \eqref{stab-VE-INTRO}, and \emph{pure jump transitions}, defined on discrete subsets of $E$, along which the stability  \eqref{stab-VE-INTRO} may be violated.
  A notable property 
 of  $\VE$ solutions is that, if an optimal jump transition $\teta:E \to \Xs$ at a jump point $t$ does not comply with the stability condition 
\eqref{stab-VE-INTRO} at some $s\in E$, then $s$ is isolated and, denoting by $s_-: = \max (E \cap ({-}\infty,s))$, there holds
\[
\teta(s) \in \mathrm{Argmin}_{y\in \Xs} \left\{ \ene ty + \md {\teta(s_-)}y+ \frac{\mu}2 \mdn^2(\teta(s_-),y)\right\}\,.
\]
A  complete characterization of $\VE$ solutions to one-dimensional rate-independent systems has been recently provided in 
\cite{Minotti16}, showing that their behavior strongly depends on the parameter $\mu$. When $\mu=0$, $\VE$ solutions coincide with Energetic solutions and therefore they satisfy the \emph{Maxwell rule}. For a sufficiently `strong' viscous correction, i.e.\ with $\mu$ above a certain threshold depending on the (nonconvex) driving energy, $\VE$ solutions exhibit a behavior akin to that of $\BV$ solutions, and follow the \emph{delay rule}. With a `weak' correction, $\VE$ solutions have an intermediate character between Energetic and $\BV$ solutions.
\subsection{Main results}
In this paper, we aim to
gain further insight into  this in-between quality of $\VE$ solutions and into the role of the 
tuning parameter $\mu$, revealed by the analysis in \cite{Minotti16}, in a more general context.  To this end, we shall study the singular limits
of  $\VEa{\mu}$  solutions to the (metric) rate-independent system $\RIS$
 as $\mu \down0$ and $\mu \uparrow \infty$. 
\par
With \textbf{Theorem \ref{th:1}} we will show that, any sequence $(u_n)_n$ of   $\VEn$ solutions corresponding to a null sequence $\mu_n \down 0$  
converges, up to a subsequence, to an Energetic solution of $\RIS$. 
\textbf{Theorem \ref{th:2}}  will address the behavior of a sequence $(u_n)_n$ of  $\VEn$  solutions with parameters $\mu_n \uparrow \infty$. In this case, in accordance with condition \eqref{eps-tau-INTRO}, we expect to obtain 
$\BV$ solutions. 
 We will prove indeed that, up to a subsequence, as $\mu_n\uparrow \infty$  $\VEn$ solutions converge  to a $\BV$ solution of $\RIS$.
\par
While referring to Sections \ref{s:3} and \ref{s:4} for further comments and all details, let us mention here that the proof of Thm.\ \ref{th:2} is quite challenging. In fact, it  involves passing 
from the  transitions that describe the jump behavior of a  sequence of  $\VEn$ solutions, and that are  given by a collection of `sliding
 pieces' and discrete trajectories, to the jump transitions for $\BV$ solutions, that  are instead \emph{absolutely continuous} curves. This can be achieved by  means of  a careful reparameterization technique, combined with a delicate compactness argument for transition curves in \emph{varying} domains.
\paragraph{\bf Plan of the paper} In Section \ref{s:2} we collect some preliminary results, set up the basic assumptions on the energy functional $\calE$, and give the precise definitions of Energetic, Balanced Viscosity, and Visco-Energetic solutions to the rate-independent system $\RIS$. In Section \ref{s:main-results} we recapitulate the existence results for the three solution concepts, and state our own Theorems \ref{th:1} and \ref{th:2}, whose proof is developed throughout Sections \ref{s:3} and \ref{s:4}, also resorting to some auxiliary results stated and proved in the Appendix.
\section{Preliminary results and overview of the   solution concepts for rate-independent systems}
\label{s:2}
We start by fixing some notation:
Given an arbitrary subset $E\subset \R$, we shall denote by 
\begin{equation}
\label{notation-set}
\begin{aligned}
\mathfrak{P}_f(E) \text{ the collection of all finite subsets of } E,
\quad
E^-: = \inf E, \quad E^+: = \sup E\,.
\end{aligned}
\end{equation}
\paragraph{\bf Kuratowski convergence of sets} In view of the compactness argument developed in Section \ref{s:4} ahead,  here we  provide a minimal aside on the notion of  \emph{Kuratowski} convergence of  sets,  confining the discussion to \emph{closed} sets,  and 
referring to
\cite{Ambrosio-Tilli} for all details.
We say that a
sequence $(C_n)_n$ of closed subsets of $\Xs$ converge in the
sense of Kuratowski to a closed set $C $,
 if
\begin{equation}
\label{kur-conv}
{\kliminf}_{n\to \infty} C_n = {\klimsup}_{n\to\infty} C_n = C,
\end{equation}
where
\begin{subequations}
\begin{align}
&
\label{Lliminf}  {\kliminf}_{n\to\infty} C_n  := \{ x \in \Xs\, : 
 \ \exists\, x_n \in C_n \text{ such that } x_n \to x
\},
\\
& 
\label{Llimsup}
  {\klimsup}_{n\to\infty} C_n  := \{ x \in \Xs\, : \
\exists\, j \mapsto n_j \text{ increasing  and } x_{n_j} \in C_{n_j} \text{ such that } x_{n_j} \to  x  \}.
\end{align}
\end{subequations}
If all the
closed sets $C_n$ are contained in a compact set $K$, then  Kuratowski convergence coincides with the convergence induced by the 
Hausdorff distance
 \cite[Prop.\ 4.4.14]{Ambrosio-Tilli}. That is why, the   \emph{Blaschke} 
 Theorem  (cf., e.g., \cite[Thm.\
4.4.15]{Ambrosio-Tilli})  is   applicable, ensuring 
 that, if $K \subset \Xs$ is a
fixed compact set, then every sequence of closed sets $(C_n)_n
\subset K$ admits a subsequence converging in the
Kuratowski sense to
a closed set $C \subset K$. If the sets $C_n$
are connected, then 
  $C$ is  also connected. 

\subsection{Preliminaries on functions of bounded variation and absolutely continuous functions}
\label{ss:2.1}
Let us first recall some preliminary definitions and properties related to functions of bounded variation with values in the metric space $(\Xs,\mdn)$. The \emph{pointwise} total variation $\Vars{\mdn}u{E}$ of a function $u:E \to \Xs$ is defined by
\begin{equation}
\label{def-tot-var}
\Vars{\mdn}u{E}: = \sup \left\{ \sum_{j=1}^M \md{u(t_{j-1})}{u(t_j)}\, : \ t_0<t_1<\ldots<t_M, \ \{ t_j\}_{j=0}^M \in \mathfrak{P}_f(E)\right\}\,,
\end{equation}
with 
$\Vars{\mdn}u{\emptyset}:=0$.
We define the space of functions with bounded variation via
\[
\BV_{\mdn}(E;X): = \{ u: E \to \Xs\, : \ \Vars{\mdn}u{E}<\infty\}\,.
\]
For every $u\in \BV_{\mdn}(E;X)$ we may introduce the function 
\begin{equation}
\label{def-variation-func}
V_u: [E^-,E^+]\to [0,\infty) \quad \text{given by} \quad
\Vf ut: = \Vars{\mdn}u{E \cap [ E^-,t]} \,.
\end{equation}
Observe that 
$V_u$ is monotone nondecreasing and satisfies
\[
\md{u(t_0)}{u(t_1)} \leq \Vari {\mdn}{u}{t_0}{t_1} = \Vf u{t_1}-  \Vf u{t_0} \quad \text{for all } t_0,t_1\in E \text{ with } t_0\leq t_1\,.
\]
\par
Since the metric space $(\Xs,\mdn)$ is complete, every function $u\in \BV_{\mdn}(E;X)$ is \emph{regulated}, i.e. at every $t\in E$ the left and right limits 
$\lli ut$ and $\rli ut$
exist (with obvious adjustments at $E^-$ and $E^+$). We recall that $u$ only has jump discontinuities, and that its (at most) \emph{countable} jump set $ \jump u$ coincides with the jump set of $V_u$. 
\par
 We will also consider the distributional derivative
 $\mum u{}$ of the function $V_u$ and recall that the Borel measure $\mum u {}$ can be decomposed into the sum
 \begin{equation}
 \label{decompose}
 \mum u{} =\mum u{\mathrm{d}} + \mum u{\mathrm{J}}  
 \end{equation}
with $\mum u{\mathrm{d}}$ the diffuse part of $\mum u{}$ (i.e.\ the sum of its absolutely continuous and Cantor parts), fulfilling $\mum u{\mathrm{d}}(\{t\})=0$ for every $t\in [E^-,E^+]$, and $\mum u{\mathrm{J}}$ its jump part, concentrated on the set $ \jump u$, so that 
\[
\mum u{\mathrm{J}}(\{t\}) = \md{\lli ut}{u(t)} +\md{u(t))}{\rli ut} \quad \text{for every } t \in \jump u.
\]
Therefore we have 
\begin{equation}
\label{Var-mu-Javr}
\Vari {\mdn}u{t_0}{t_1} = \mum u{\mathrm{d}}([t_0,t_1])+ \Jvar {\mdn}{u}{t_0}{t_1}
\end{equation}
for every interval
$[t_0,t_1]\subset E$,
with the jump contribution  
\begin{equation}
\label{d-jump-contribution}
 \Jvar {\mdn}{u}{t_0}{t_1}: =  \md{u(t_0))}{\rli u{t_0}} +  \md{\lli u{t_1})}{ u(t_1)} + \sum_{t\in \jump u \cap (t_0,t_1)} \left( \md{\lli ut}{u(t)} +\md{u(t)}{\rli ut} \right)\,.
\end{equation}
\par
In the definition of Balanced Viscosity and Visco-Energetic solutions, there will come into play an alternative notion of total variation for a curve $u\in \BV([0,T];\Xs)$, which will reflect the energetic behavior of the (Balanced Viscosity/Visco-Energetic) solution at jump points. 
It will  be obtained by suitably modifying  the jump contribution to the total variation induced by $\mdn$, cf.\ \eqref{Var-mu-Javr},  in terms of  a (general) \emph{cost function}
$\gcostname: [0,T] \times \Xs\times \Xs \to [0,\infty]$,  with $\gcostname \geq \mdn$, 
that shall measure the energy dissipated along a jump. Thus,   hereafter we will refer to $\gcostname$ as \emph{jump dissipation cost}.
 As particular cases of $\gcostname$, we will consider   
 \begin{itemize}
 \item[-] the 
 \emph{viscous (jump dissipation) cost}   $\bvcostname$, cf.\ \eqref{vcost} ahead, in the case of Balanced Viscosity solutions;
 \item[-]
  the \emph{visco-energetic (jump dissipation)  cost} $\vecostname$, cf.\ \eqref{vecost} ahead, in the case of Visco-Energetic solutions.
  \end{itemize}
  \par
With the jump dissipation cost $\gcostname$ we associate the  \emph{incremental cost}
\begin{equation}
\label{Delta-e}
\Delta_{\gcostname}: [0,T]\times \Xs \times \Xs \to [0,\infty], \quad \Delta_{\gcostname}(t,u_-,u_+): = \gcost  t{u_-}{u_+} - \md{u_-}{u_+}
 \text{ for all } t\in[0,T], \, u_-,\, u_+ \in \Xs,
\end{equation}
where the notation $u_-,\, u_+$  is suggestive of the fact that, in the definition of the total variation functional induced by $\gcostname$, the incremental cost will be evaluated at the left and right limits $\lli u t$ and $\rli u t$ at a jump point of a curve $u$. 
We will also use the notation
\[
\Delta_{\gcostname}(t,u_-,u,u_+): = \Delta_{\gcostname}(t,u_-,u) + \Delta_{\gcostname}(t,u,u_+)\,.
\]
We are now in a position to introduce the \emph{augmented total variation} functional induced by $\gcostname$.
\begin{definition}
\label{def-Var-e}
Given a (jump dissipation) cost function $\gcostname$ and the associated incremental cost $\Delta_{\gcostname}$, and   given a  curve $u\in \BV([0,T];\Xs)$, we define the \emph{incremental jump variation} of $u $ on a sub-interval $[t_0,t_1] \subset [0,T]$ by 
\begin{equation}
\label{jump-Delta-e}
 \Jvar {\Delta_{\gcostname}}{u}{t_0}{t_1} : = \idelta{\gcostname}{t_0}{u(t_0))}{\rli u{t_0}} +  \idelta{\gcostname}{t_1}{\lli u{t_1})}{ u(t_1)} + \sum_{t\in \jump u \cap (t_0,t_1)} 
 \aidelta{\gcostname}{t}{\lli u{t}}{u(t)}{\rli u{t}}\,.
 \end{equation}
 This induces the \emph{augmented total variation} functional
 \begin{equation}
 \label{augm-tot-var}
 \Vari {\mdn,\gcostname}u{t_0}{t_1} : = \Vari {\mdn}u{t_0}{t_1} +  \Jvar {\Delta_{\gcostname}}{u}{t_0}{t_1} \quad\text{along any sub-interval } [t_0,t_1]\subset [0,T].
\end{equation}
\end{definition}
\noindent
Since we have subtracted from the $\gcostname$-jump contribution the $\mdn$-distance of the jump end-points, cf.\ \eqref{Delta-e}, the $\mdn$-jump contribution \eqref{d-jump-contribution} to  $\Varname {\mdn}$ cancels out, and in fact only the \emph{diffuse} contribution $ \mum u{\mathrm{d}}([t_0,t_1])$ remains. In fact, one could rewrite $\Vari {\mdn,\gcostname}u{t_0}{t_1}$ as
\begin{equation}
\label{serve?}
\Vari {\mdn,\gcostname}u{t_0}{t_1} =  \mum u{\mathrm{d}}([t_0,t_1]) +  \Jvar {{\gcostname}}{u}{t_0}{t_1},
\end{equation}
with $ \Jvar {{\gcostname}}{u}{t_0}{t_1}$ defined by \eqref{jump-Delta-e} with  the ``whole'' cost $\gcostname$ in place of its incremental version $\Delta_{\gcostname}$, i.e.\
\begin{equation}
\label{jump-Delta-e-simpler}
\begin{aligned}
 \Jvar {{\gcostname}}{u}{t_0}{t_1} : = \gcost{t_0}{u(t_0)}{\rli u{t_0}}  & +  \gcost{t_1}{\lli u{t_1}}{ u(t_1)} \\ & + \sum_{t\in \jump u \cap (t_0,t_1)} 
\left(  \gcost{t}{\lli u{t}}{u(t)} + \gcost{t}{u(t)} {\rli u{t}} \right) \,.
\end{aligned}
\end{equation}
\par
Clearly, 
$\Vari {\mdn,\gcostname}u{t_0}{t_1} \geq \Vari {\mdn}u{t_0}{t_1}$, and they coincide if $\gcostname = \mdn$, or when $\jump u =\emptyset$. Moreover,   as observed in 
\cite{SavMin16},  although it need not be induced by a distance on $\Xs$,
$\Varname{\mdn,\gcostname} $  still enjoys the additivity property
\[
 \Vari {\mdn,\gcostname}u{a}{c} =  \Vari {\mdn,\gcostname}u{a}{b} +  \Vari {\mdn,\gcostname}u{b}{c} \qquad \text{for all } 0\leq a \leq  b \leq c \leq T.
\]
\par
Finally,
we recall that  a 
curve $u: [0,T]\to \Xs$ is absolutely continuous 
(and write $u\in \AC([0,T];\Xs)$)
if there exists $m\in L^1(0,T)$ such that
\begin{equation}\label{metric_dev}
\md{u(s)}{u(t)}\leq \int_s^t m(r)\dd r \quad \text{for all $0\leq s\leq
t\leq T.$}
\end{equation}
For every  $u
\in\AC([0,T];\Xs) $, the limit
\begin{equation}
\label{m-derivative}
|u'|(t) = \lim_{s\to t}\frac{\mdn(u(s),u(t))}{|t-s|}
\quad \text{exists for a.a. } t\in (0,T),
\end{equation}
cf.\ \cite[Sec.\ 1.1]{AGS08}.
 We will refer to it as the {\it metric
derivative} of $u$ at $t$. The map
 $t \mapsto|u'|(t)$  belongs to $ L^1(0,T) $
and it is minimal within the class of functions $m\in L^1(0,T)$
fulfilling \eqref{metric_dev}. 
\subsection{Energetic, Balanced Viscosity, and Visco-Energetic solutions at a glance}
\label{ss:2.2}
We now give a quick overview of the notions of rate-independent evolution this paper is concerned with.
We aim to  somehow motivate the various solution concepts and in addition
 highlight both the common points, and the differences,  in their structure.
 \par
  Underlying the upcoming definitions, there will be the following
  basic conditions on the energy functional $\calE$. Let us mention in advance 
that we in fact allow for a possibly nonsmooth time-dependence $t\mapsto \ene tu$.
 However, in what follows for simplicity we will confine our analysis to the case in which the  domain of $\ene t{\cdot} $  in fact coincides with $\Xs$ for every $t\in [0,T]$, referring to \cite[Rmk.\ 2.7]{SavMin16} for a discussion of the more general case in which $\mathrm{dom}( \ene t{\cdot} )$ is a proper subset of $\Xs$ (still independent of  the time variable). 
 \paragraph{\bf Basic assumptions on the energy}
Throughout the paper, we will require that  $\calE$ complies with two basic properties, involving the perturbed energy functional
\begin{equation}
\pertn: [0,T]\times \Xs \to \R \quad \pert tu: = \ene tu + \md{x_o}u \quad \text{with $x_o$ a given reference point in $X$}
\end{equation}
and its sublevel sets $\subl C: = \{ (t,u)\in [0,T]\times \Xs\, : \ \pert tu \leq C\}$. Namely,
\begin{itemize}
\item[\textbf{Lower semicontinuity and compactness:}] for all $C\in \R$
\begin{equation}
\label{Ezero}
\tag{$\mathrm{E}_1$}
\text{$\calE$ is lower semicontinuous on $\subl C  $ and the sets $\subl C$ are compact in $[0,T]\times X$};
\end{equation}
\item[\textbf{Power control:}]  there exists a map $\pwn : [0,T]\times \Xs \to \R$ fufilling 
\begin{equation}
\label{Power}
\tag{$\mathrm{E}_2$}
\begin{gathered}
\liminf_{s\up t} \frac{\ene tu -\ene su}{t-s} \geq \pw tu \geq \limsup_{s\down t} \frac{\ene su -\ene tu}{s-t} \quad \text{for all } (t,u) \in [0,T]\times \Xs,
\\
\exists\, C_P>0 \ \  \forall\, (t,u) \in [0,T]\times \Xs\, : \quad
|\pw tu | \leq C_P \pert tu\,.
\end{gathered}
\end{equation}
\end{itemize}
We may understand the power functional $\pwn$ as a sort of ``time superdifferential'' of the energy functional, surrogating its partial time derivative 
in the case where the functional  $t\mapsto \ene tu $ is not differentiable  at every point of $[0,T]\times \Xs$. This for instance occurs for reduced energies having the form $\ene tu = \min_{\varphi \in \Phi} \calI(t,\varphi, u)$ and such that the set of minimizers 
does not reduce to a singleton,
as considered, e.g., in \cite{KMZ10-poly,MRS2013,MRS-MJM,SavMin16}.
 By repeating the very same arguments as in \cite{SavMin16}, we may deduce from \eqref{Ezero} \& \eqref{Power} that
\begin{equation}
\label{Lip-cont-E}
\begin{gathered}
\text{
the function $t\mapsto \ene tu $ is Lipschitz continuous  for every $u\in \Xs$, with}
\\
\pw tu = \partial_t \ene tu \qquad \text{for almost all } t \in [0,T] \text{ and for all } u \in \Xs.
\end{gathered}
\end{equation}
Therefore, 
\begin{equation}
\label{2tfc}
\ene tu = \ene su +\int_s^t \pw ru \dd r  \qquad \text{for every $[s,t]\subset[0,T]$.}
\end{equation}
 Combining this with the power control estimate in 
\eqref{Power} and exploiting the Gronwall Lemma, we conclude that 
\begin{equation}
\label{prop-pert}
\pert tu\leq \pert su \exp\left(C_P |t-s|\right) \quad \text{for all } s,\, t \in [0,T].
\end{equation}
That is why, it is significant (and notationally convenient) to work with the functional $\perto u : = \pert 0u$, which controls $\pert tu$, and thus the power functional $\pw tu$,  at all $t\in [0,T]$. 
 \par
 We are now in a position to give  the 
  concept of \underline{\textbf{Energetic}} solution, dating back to 
  \cite{MieThe99MMRI,MieThe04RIHM}, cf.\ also \cite{Miel05ERIS}.  
\begin{definition}[Energetic solution]
\label{def:en-sol}
A curve $u\in \BV([0,T];X)$ is an Energetic solution of the rate-independent system $\RIS$ if it satisfies for every $t\in [0,T]$
\begin{itemize}
\item[-] the global stability condition
\begin{equation}
\label{global-stab}
\tag{$\mathrm{S}_{\mdn}$}
\ene t{u(t)} \leq \ene tv + \md{u(t)}{v} \quad \text{for every } v \in \Xs,
\end{equation}
\item[-]  the energy balance
\begin{equation}
\label{global-enbal}
\tag{$\mathrm{E}_{\mdn}$}
\ene t{u(t)} + \Vari {\mdn}{u}0{t} = \ene 0{u(0)} +\int_0^t \pw s{u(s)} \dd s \,.
\end{equation}
\end{itemize}
\end{definition}
For later use, we introduce the $\mdn$-stable  set 
\[
\stab \mdn: = \{ (t,u) \in [0,T]\times \Xs \, : \ \ene tu \leq \ene tv + \md uv  \text{ for all } v \in \Xs\}, 
\]
 with its time-dependent sections $\stabi \mdn t : = \{ u \in \Xs\, : \ (t,u) \in \stab \mdn \}\,$. 
We postpone to Section \ref{ss:2.3} a discussion on the existence of Energetic solutions. 
\par
As already mentioned in the Introduction, 
\underline{\textbf{Balanced Viscosity}} solutions arise in the time-continuous limit of the time-incremental scheme \eqref{tims-bv}, when the 
 parameters $\eps$ and $\tau$ both tend to zero with  $\tfrac{\eps}\tau \up \infty$ cf.\ \eqref{eps-tau-INTRO}.
They
 fulfill the \emph{local version} of the stability condition \eqref{global-stab}, involving the \emph{metric slope} of the energy functional $\calE$, cf.\  \eqref{slope-INTRO}.
The ``viscous'' character  of the approximation that underlies condition
\eqref{eps-tau-INTRO},
is also reflected in the \emph{viscous jump dissipation cost}. Indeed, at fixed process time $t\in [0,T]$,  $ \bvcost t{u_-}{u_+}$  is 
obtained by minimizing the \emph{transition cost}
\begin{equation}
\label{bv-trans-cost}
\tcost{\BV}{t}{\teta}{[r_0,r_1]}: = \int_{r_0}^{r_1}
|\teta'|(r) \left( \slope \calE t{\teta(r)} \vee 1 \right)  \dd r 
\end{equation}
 over all \emph{absolutely continuous} curves $\teta$ on an interval $[r_0,r_1]$, connecting the two points $u_-$ and $u_+$, where we recall that  $ |\teta'|$ is the (almost everywhere defined) \emph{metric derivative} of the  curve $\teta$. 
 Namely,
 \begin{equation}
 \label{vcost}
 \bvcost t{u_-}{u_+}: = \inf\left\{ \tcost{\BV}{t}{\teta}{[r_0,r_1]}
   : \,  \teta \in \AC([r_0,r_1];\Xs),  \ \teta(r_0)=u_-, \ \teta(r_1) = u_+\right\}\,.
\end{equation}
We can then introduce the incremental cost   $\Delta_{\bvcostname}$ \eqref{Delta-e} and the  jump variation $\mathrm{Jmp}_{\Delta_{\bvcostname}}$ \eqref{jump-Delta-e} associated with $\bvcostname$, and thus arrive at the induced augmented total variation $\Varname{\mdn,\bvcostname}$ \eqref{augm-tot-var}, which enters into the energy balance involved in the Balanced Viscosity concept.
\begin{definition}[Balanced Viscosity solution]
\label{def:bv-sol}
A curve $u\in \BV([0,T];X)$ is a Balanced Viscosity ($\BV$) solution of the rate-independent system $\RIS$
if it satisfies
\begin{itemize}
\item[-] the local stability condition
\begin{equation}
\label{loc-stab}
\tag{$\mathrm{S}_{\mdn,\mathrm{loc}}$}
\slope \calE t{u(t)} \leq 1 \quad \text{for every } t \in [0,T]\setminus \jump u,
\end{equation}
\item[-]  the energy balance
\begin{equation}
\label{BV-enbal}
\tag{$\mathrm{E}_{\mdn,\bvcostname}$}
\ene t{u(t)} + \Vari {\mdn,\bvcostname}{u}0{t} = \ene 0{u(0)} +\int_0^t \pw s{u(s)} \dd s \quad \text{for all } t \in [0,T]\,.
\end{equation}
\end{itemize}
\end{definition} 
\par
The notion of \underline{\textbf{Visco-Energetic}} solution features a modified concept of stability which also involves the viscous correction $\corr uv = \tfrac \mu2 \mdn^2(u,v)$.
We then define the functional
\begin{equation}
\label{cmdn}
\cmd uv: = \md u v + \corr uv = \md u v +\frac \mu2 \mdn^2(u,v)
\end{equation}
and 
we say that a point $(t,x)\in [0,T]\times \Xs$ is $\cmdn$-stable if \begin{equation}
\label{cmdn-stab}
\ene tx \leq \ene ty + \cmd xy = \ene ty +\md xy +\frac\mu2 \mdn^2(x,y)  \qquad \text{for all } y \in \Xs\,.
\end{equation}
We denote by $\stab \cmdn$ the collection of all $\cmdn$-stable points, and by $\stabi \cmdn t$ its section at time $t\in [0,T]$. We also introduce the \emph{residual stability function} $\calR:[0,T]\times \Xs \to \R$ given by 
\begin{equation}
\label{residual-stability-function}
\rstab tx: = \sup_{y\in \Xs} \left\{ \ene tx - \ene ty -\cmd xy \right\} = \ene tx - \inf_{y\in \Xs} \left\{ \ene ty+\cmd xy\right\}
\end{equation}
(for simplicity, we choose to neglect the $\mu$-dependence of the functionals $\cmdn$ and $\calR$ in their notation). 
Observe that
\begin{equation}
\label{propR}
\rstab tx \geq 0 \quad \text{for all } (t,x) \in [0,T]\times \Xs \quad \text{with} \quad \rstab tx=0 \text{ if and only if } (t,x) \in \stab \cmdn\,,
\end{equation}
so that $\calR$ may be interpreted as 
``measuring the failure'' of the stability condition at a given point  $(t,x) \in [0,T]\times \Xs$.
It can be straightforwardly checked that,  under the basic lower semicontinuity assumption \eqref{Ezero} on $\calE$, the functional $\calR$ is lower semicontinuous on $[0,T]\times \Xs$. 
\par
We now have all the ingredients to define the jump-dissipation cost for Visco-Energetic solutions. In the same way as for Balanced Viscosity solutions, such a  cost is obtained by minimizing a suitable transition cost over a class of curves connecting the two end-points of the jump. 
However, such curves, while still continuous, need not be absolutely continuous. Further, they are in general defined on a \emph{compact} subset $E\subset\R$ that may have  a  more complicated structure than that of an interval. To describe it, we 
introduce
\begin{equation}
\label{holes}
\text{the collection $\hole E$ of the connected components of the  set $[E^-,E^+]\setminus E$,}
\end{equation}
where we recall that $E^-=\inf E$ and $E^+=\sup E$. Since $[E^-,E^+]\setminus E$ is an open set, $\hole E$ consists of at most countably many open intervals, which we will often refer to as the ``holes'' of $E$. 
Hence, the transition cost at the basis of the concept of Visco-Energetic solution evaluates 
(1)  the  $\mdn$-total variation of a continuous curve defined on a set $E$, (2) the sum, over all the holes of $E$, of a quantity  related to the gaps 
(3)  the measure of ``how much''  the curve $\teta$ fails to comply with the $\cmdn$-stability condition \eqref{cmdn-stab} at the points in $E \setminus \{E^+\}$.
\begin{definition}
\label{def-VE-trans-cost}
Let $E$ be a compact subset of $\R$ and $\teta \in \mathrm{C}(E;\Xs)$. For every $t\in [0,T]$ we define the \emph{transition cost function} 
\begin{equation}
\label{ve-tcost}
\tcost{\VE}{t}{\teta}{E} : = \Vars {\mdn}\teta E + \Gap{\mdn}\teta E  + \sum_{s\in E{\setminus}E^+} \rstab t{\teta(s)}\,,
\end{equation}
with
\begin{enumerate}
\item $ \Vars {\mdn}\teta E $ from \eqref{def-tot-var};
\item $\Gap{\mdn}\teta E: = \sum_{I\in \hole E} \frac\mu2 \mdn^2(\teta(I^-),\teta(I^+))$;
\item the (possibly infinite) sum
\[
 \sum_{s\in E{\setminus}E^+} \rstab t{\teta(s)}: =  \begin{cases}
 \sup \{  \sum_{s\in P} \rstab t{\teta(s)}\, :  \ P \in \mathfrak{P}_f(E) \} &\text{ if }  E{\setminus}E^+ \neq \emptyset,
 \\
 0 &\text{ otherwise}
 \end{cases}
\]
(recall that $ \mathfrak{P}_f(E)$ denotes the collection of all finite subsets of $E$).
\end{enumerate}
\end{definition}
Along with \cite{SavMin16}, we observe that, for every fixed $t\in [0,T]$ and $\teta \in \mathrm{C}(E;\Xs)$, the transition cost fulfills the additivity property
\[
\tcost{\VE}{t}{\teta}{E \cap [a,c]} = \tcost{\VE}{t}{\teta}{E \cap [a,b]} + \tcost{\VE}{t}{\teta}{E \cap [b,c]}  \quad \text{for all } a<b<c\,.
\]
We are now in a position to define the associated \emph{visco-energetic jump dissipation cost} $\vecostname: [0,T]\times X \times X \to [0,\infty]$ via
 \begin{equation}
 \label{vecost}
  \vecost t{u_-}{u_+}: = \inf\{ \tcost{\VE}{t}{\teta}{E}\, : \ E \Subset \R, \ \teta \in \mathrm{C}(E;\Xs), \ \teta(E^-) =u_-, \  \teta(E^+) =u_+ \},
\end{equation}
whence the incremental dissipation cost $\Delta_{\vecostname}$ according to   \eqref{Delta-e}, the jump variation $\mathrm{Jmp}_{\Delta_{\vecostname}}$ as in  \eqref{jump-Delta-e}, and the  augmented total variation $\Varname{\mdn,\vecostname}$ as in  \eqref{augm-tot-var}. 
\par
We can now give the following
\begin{definition}[Visco-Energetic solution]
A curve $u\in \BV([0,T];X)$ is  a Visco-Energetic ($\VE$) solution of the rate-independent system $\RIS$
if it satisfies
\begin{itemize}
\item[-] the $\cmdn$-stability condition
\begin{equation}
\label{stab-VE}
\tag{$\mathrm{S}_{\cmdn}$}
\ene t{u(t)} \leq \ene tv + \md{u(t)}v + \frac\mu2 \mdn^2 (u(t),v)   \quad \text{for every } v \in \Xs \text{ and for every } t \in [0,T]\setminus \jump u,
\end{equation}
\item[-]  the energy balance
\begin{equation}
\label{enbal-VE}
\tag{$\mathrm{E}_{\mdn,\vecostname}$}
\ene t{u(t)} + \Vari {\mdn,\vecostname}{u}0{t} = \ene 0{u(0)} +\int_0^t \pw s{u(s)} \dd s \quad \text{for all } t \in [0,T]\,.
\end{equation}
\end{itemize}
\end{definition}
\section{Main results}
\label{s:main-results}
Prior to stating our own results on the singular limits of $\VE$ solutions in Section \ref{ss:2.4}, 
in Sec.\ \ref{ss:2.3} below 
we recall the known existence results for Energetic,  $\BV$,   and $\VE$ solutions. Under the same conditions ensuring the existence for the two former solution concepts, we will prove our convergence statements 
for $\VEa{\mu}$ solutions 
 in the limits $\mu\down0$ and $\mu\up\infty$, respectively. 
\subsection{A survey on existence  results}
\label{ss:2.3}
In what follows, in addition to the basic conditions   \eqref{Ezero}  and  \eqref{Power},  we will introduce further assumptions 
on the energy functional $\calE$
that will be at the core of the upcoming existence results for  Energetic (Thm.\ \ref{th:ex-en}), $\BV$ (Thm.\ \ref{th:ex-bv}), and $\VE$ (Thm.\ \ref{th:ex-ve}) solutions. We will also 
illustrate the main ideas underlying their proofs.
\paragraph{\bf Energetic solutions.} For the existence of Energetic solutions in the metric setting of \eqref{met-sp} we refer to \cite[Thm.\ 4.5]{MaiMie05EREM}, cf.\ also \cite{Miel05ERIS} and \cite[Sec.\ 2.1]{MieRouBOOK}.  In accordance with these results,
in addition  to the coercivity \eqref{Ezero} and the   power control \eqref{Power},   we require that
\begin{itemize}
\item[\textbf{Upper semicontinuity of the power:}]
$\pwn :[0,T]\times \Xs \to \R$ satisfies the \emph{conditional upper semicontinuity} condition
\begin{equation}
\label{uscPower}
\tag{$\mathrm{E}_3$}
\left( (t_n,u_n)\to(t,u) \text{ in } [0,T]\times \Xs, \ \ene {t_n}{u_n}\to \ene tu \right) \ \Longrightarrow \ \limsup_{n\to\infty} \pw {t_n}{u_n}\leq \pw tu\,.
\end{equation}
\end{itemize}
  We thus have 
  \begin{theorem}
  \label{th:ex-en}
  Let $\calE: [0,T]\times \Xs \to \R$ comply with  \eqref{Ezero}, \eqref{Power} and \eqref{uscPower}. Then, for every  initial datum $u_0$ stable at $t=0$, i.e.\ $u_0 \in \stabi \mdn 0$, there exists at least one Energetic solution to the rate-independent system $\RIS$ with $u(0)=u_0$.
  \end{theorem}
  The \emph{proof} is based on a (by now standard in the frame of rate-independent systems)
  time-discretization procedure, with the discrete solutions 
  constructed by recursively solving the time-incremental minimization scheme \eqref{tims-ris}. 
  Their (piecewise constant) interpolants are shown to comply with the discrete versions of the stability condition \eqref{global-stab} and of the upper energy estimate in \eqref{global-enbal}, whence all a priori estimates  stem, also based on the power control \eqref{Power}. With a Helly-type compactness result, crucially relying on   \eqref{Ezero}, we  thus infer that the approximate solutions pointwise converge to a curve $u\in \BV([0,T];X)$. 
  The continuity
  (cf.\ \eqref{Lip-cont-E}) and 
    lower semicontinuity   
  properties   \begin{equation}
\label{4closure-stable-set}
t_n \to t \ \Rightarrow \   \ene {t_n}y \to \ene ty \text{ for all } y\in \Xs, \qquad \left( t_n  \to t, \ u_n \to u \right) \ \Rightarrow \ \liminf_{n\to\infty} 
 \ene {t_n}{u_n} \geq \ene tu 
  \end{equation}
  ensure the closedness of the stable set $\stab \mdn$, which
  allows us to pass to the limit in the discrete stability condition and conclude that $u$ complies with  \eqref{global-stab}. Lower semicontinuity arguments, joint with \eqref{uscPower}, lead to the limit passage in the discrete upper energy estimate, so that $u$ complies with the  upper energy estimate  $\leq$  of  \eqref{global-enbal}. The lower energy estimate $\geq$ can be then deduced from the stability condition either via a Riemann-sum argument, formalized in, e.g.,    \cite[Prop,\ 2.1.23]{MieRouBOOK}, or by applying \cite[Lemma 6.2]{SavMin16}. 
\paragraph{\bf Balanced Viscosity solutions.} 
Along the footsteps of 
\cite[Thm.\ 4.2]{MRS-MJM}, 
 for the existence of Balanced Viscosity solutions, in addition to  \eqref{Ezero} and \eqref{Power},  we again need to impose the (conditional) upper semicontinuity of the power functional and, \emph{in addition}, the 
 lower semicontinuity of the slope
 along sequences \emph{with bounded energy and  slope}.
  These requirements are subsumed by the following condition:
\begin{itemize}
\item[\textbf{Upper semicontinuity of the power, lower semicontinuity of the slope:}] 
$\calE :[0,T]\times \Xs \to \R$ and
$\pwn :[0,T]\times \Xs \to \R$ satisfy
\begin{equation}
\label{uscPower-bis}
\tag{$\mathrm{E}_3'$}
\begin{aligned}
&
\left( (t_n,u_n)\to(t,u) \text{ in } [0,T]\times \Xs, \ 
\sup_{n\in\N} \perto{u_n}<\infty, \ \sup_{n\in \N} \slope \calE{t_n}{u_n}<\infty
 \right) 
\\
&  \Longrightarrow \ \begin{cases}
\liminf_{n\to\infty} \slope \calE{t_n}{u_n} \geq \slope \calE{t}{u},
\\
 \limsup_{n\to\infty} \pw {t_n}{u_n}\leq \pw tu\,.
 \end{cases}
 \end{aligned}
 \end{equation}
\end{itemize}
The last, key condition  underlying the existence of Balanced Viscosity solutions is that $\calE$ complies with the
\begin{itemize}
\item[\textbf{Chain-rule inequality:}] for every curve $u\in \AC([0,T];\Xs)$   the function $t\mapsto \ene t{u(t)}$ is absolutely continuous on $[0,T]$, and there holds
\begin{equation}
\label{ch-rule-ineq}
\tag{$\mathrm{E}_4$}
-\frac{\dd}{\dd t} \ene t{u(t)} +\pw t{u(t)} \leq |u'|(t) \slope \calE{t}{u(t)} \qquad \foraa t \in (0,T)\,.
\end{equation}
\end{itemize}
Under these conditions, the following existence result  was proved in \cite{MRS-MJM}.
\begin{theorem}
\label{th:ex-bv}
  Let $\calE: [0,T]\times \Xs \to \R$ comply with  \eqref{Ezero}, \eqref{Power}, \eqref{uscPower-bis}, and  \eqref{ch-rule-ineq}.
  Then, for every $u_0\in \Xs$ there exists at least one Balanced Viscosity solution to the rate-independent system $\RIS$ with $u(0)=u_0$.
\end{theorem}
As  mentioned in the Introduction, in the \emph{proof} of 
 \cite[Thm.\ 4.2]{MRS-MJM} (cf.\ also \cite{MRS09}), $\BV$ solutions arise 
 by taking  the vanishing-viscosity limit, as $\eps \down 0$,  of  the time-continuous solutions of the 
 Gradient Systems $(X,\calE,\mdn,\psi_\eps)$ with $\psi_\eps$ from \eqref{psi-eps-intro}.  
 Nonetheless, exploiting the arguments from \cite{MRS12, MRS13} in the Banach setting,
  the  vanishing-viscosity analysis  developed in \cite{MRS-MJM}  could be easily adapted to the direct limit passage in the time-discretization scheme \eqref{tims-bv}.
  In fact, the lower semicontinuity of the slope from \eqref{uscPower-bis} serves
   to the purpose of passing to the limit in the dissipation term  
   in the discrete energy-dissipation inequality arising from the scheme \eqref{tims-bv}.
  This leads to the total variation term $\Vari{\mdn,\bvcostname}{u}0t$ in the energy balance \eqref{BV-enbal}. 
   Instead, the upper semicontinuity of the power allows us to take the limit in the power term of the 
   discrete energy inequality.  In this way, it is possible to conclude that 
 any limit curve $u\in \BV([0,T];\Xs)$ of the discrete solutions complies with the local stability condition 
 \eqref{loc-stab} 
   and with   the  upper energy estimate 
 \begin{equation}
 \label{en-uee-BV}
 \tag{$\mathrm{E}_{\mdn,\bvcostname}^{\mathrm{ineq}}$}
 \ene t{u(t)} + \Vari {\mdn,\bvcostname}{u}0{t} \leq \ene 0{u(0)} +\int_0^t \pw s{u(s)} \dd s\,.
 \end{equation} 
Unlike the case of Energetic solutions, where the validity of  global stability condition \eqref{global-stab}  was sufficient to conclude  the  lower energy estimate for \eqref{global-enbal}, \eqref{loc-stab} is  not strong enough to lead to the converse inequality of \eqref{en-uee-BV}. This is instead ensured by a chain-rule argument based on \eqref{ch-rule-ineq}, cf.\  \cite[Prop.\ 4.2, Thm.\ 4.3]{MRS12}. 
\par
Finally, let us mention that,
under the very  assumptions for the existence Thm.\  \ref{th:ex-bv},
 trivially adapting the argument for \cite[Thm.\ 3.15]{MRS13} it can be shown that 
a curve $u\in \BV([0,T];\Xs)$ is a $\BV$ solution to the rate-independent system $\RIS $ if and only if it satisfies  \eqref{loc-stab}, the localized energy inequality 
\begin{equation}
\label{localized-enineq}
\ene t{u(t)} + \Vari {\mdn}{u}s{t} \leq \ene s{u(s)} +\int_s^t \pw r{u(r)} \dd r \quad \text{for all } 0 \leq s \leq t \leq T,
\end{equation}
and the jump conditions 
\begin{equation}
\begin{aligned}
\label{bv-jump-cond}
&
\ene t{\lli u t } - \ene t{u(t)} = \bvcost t{\lli ut}{u(t)},
\\
&
\ene t{ u (t) } - \ene t{\rli u t } = \bvcost t{u(t)}{\rli u t},
\\
&
\ene t{\lli u t } - \ene t{\rli ut}  = \bvcost t{\lli ut}{\rli u t}\,.
\end{aligned}
\end{equation}

\paragraph{\bf Visco-Energetic solutions.} 
As already hinted, Visco-Energetic solutions were introduced
in \cite{SavMin16} within a more complex topological setting, featuring an \emph{asymmetric} distance and a topology $\sigma$, involved in the coercivity condition on the energy functional.
 It turns out that, in the present metric setting where $\sigma$ is the topology induced by $\mdn$,
\eqref{Ezero}, \eqref{Power} and  \eqref{uscPower-bis} coincide with the conditions required on the energy functional $\calE$ within  \cite[Assumption \textbf{$<A>$}, Sec.\ 2.2]{SavMin16}.  Furthermore, the particular choice $\corr uv = \tfrac \mu2 \mdn^2(u,v)$ for the viscous correction ensures the validity of \cite[Assumption \textbf{$<B>$}, Sec.\ 3.1]{SavMin16}.
In particular, condition 
\cite[\textbf{$<B.3>$}, Sec.\ 3.1]{SavMin16}
is fulfilled, namely $\cmdn$-stability implies local $\mdn$-stability, as it can be straightforwardly checked. 
Finally, thanks to the lower semicontinuity of the residual functional $\calR$ from \eqref{residual-stability-function}, also  \cite[Assumption \textbf{$<C>$}, Sec.\ 3.3]{SavMin16} is fulfilled. Therefore, \cite[Thm.\ 3.9]{SavMin16} applies, ensuring the convergence of the 
time-incremental scheme \eqref{tims}, with $\mu>0$ fixed,  to a Visco-Energetic solution. In particular, we have the following existence result, under the \emph{same} conditions on the energy functional as in  the existence Thm.\ \ref{th:ex-en} for Energetic solutions. 
\begin{theorem}
\label{th:ex-ve}
  Let $\calE: [0,T]\times \Xs \to \R$ comply with  \eqref{Ezero}, \eqref{Power} and \eqref{uscPower}. Then, for every
   $\mu>0$ and every 
    initial datum $u_0 \in \Xs$ there exists at least one   $\VEa{\mu}$ solution  to the rate-independent system $\RIS$ with $u(0)=u_0$.
\end{theorem}
The  outline of the existence argument is the same as for   Energetic solutions, though the technical difficulties attached to the single steps are peculiar of the Visco-Energetic case. The $\cmdn$-stability condition \eqref{stab-VE} and the upper energy estimate in \eqref{enbal-VE} are derived by passing to the limit in their discrete versions, valid for the discrete solutions to the time-incremental scheme \eqref{tims}. As shown in \cite[Thm.\ 6.5]{SavMin16}, the lower energy estimate  can then be derived from   \eqref{stab-VE} 
by applying \cite[Lemma 6.2]{SavMin16}.
\par
 Under the same conditions as for the existence Thm.\ \ref{th:ex-ve}, we have the following `stability' result 
 for $\VE$ solutions with respect to convergence of the parameters $\mu_n$ to some \emph{strictly positive} $\mu$.
 \begin{proposition}
 \label{prop:added} 
   Let $\calE: [0,T]\times \Xs \to \R$ comply with  \eqref{Ezero}, \eqref{Power} and \eqref{uscPower}.   Let $(\mu_n) \subset$ fulfill 
   \[
   \mu_n\to \mu>0 \qquad \text{as $n\to\infty$}.
   \]
   Let $(u_n^0)_n,\, u_0 \subset \Xs$ fulfill 
    \begin{equation}
  \label{conv+en-conv_init}
  u_n^0 \to u_0 \quad \text{and} \quad  \ene 0{u_n^0} \to \ene 0{u_0} \text{ as $n\to\infty$}.
  \end{equation}
   \par
     Then,  there exist  a subsequence $(u_{n_k})_k$  and a curve $u\in \BV([0,T];\Xs)$ such that $u(0)=u_0$,
 \begin{equation}
\label{ptw+en-conv}
 u_{n_k}(t) \to u(t)  \ \text{ and } \ \ene {t}{u_{n_k}(t)} \to \ene t{u(t)} \quad \text{for every } t \in [0,T],
\end{equation}
and $u$ is a $\VE_\mu$ solution to the rate-independent system $\RIS$. 
 \end{proposition}
 We will outline the \emph{proof} of Proposition \ref{prop:added} at the end of Sec.\ \ref{ss:4.1}. 
\par
We conclude this section by recalling that,  $\VE$ solutions as well can be characterized in terms of  suitable jump conditions. Namely, it was proved in \cite[Prop.\ 3.8]{SavMin16} that a curve $u\in \BV([0,T];\Xs)$ is a $\VE$ solution to the rate-independent system $\RIS$ if and only if it satisfies \eqref{stab-VE}, the energy-dissipation inequality 
\eqref{localized-enineq},
and the jump conditions 
\begin{equation}
\begin{aligned}
\label{ve-jump-cond}
&
\ene t{\lli u t } - \ene t{u(t)} = \vecost t{\lli ut}{u(t)},
\\
&
\ene t{ u (t) } - \ene t{\rli u t } = \vecost t{u(t)}{\rli u t},
\\
&
\ene t{\lli u t } - \ene t{\rli ut}  = \vecost t{\lli ut}{\rli u t}\,.
\end{aligned}
\end{equation}

\subsection{Main results:   Singular limits of  Visco-Energetic solutions}
\label{ss:2.4}
We now consider a sequence $(\mu_n)_n\subset (0,\infty)$, either converging to $0$, or diverging to $\infty$.
 Accordingly, let $(u_n^0)_n\subset \Xs$ be  a sequence of initial data for  the rate-independent system $\RIS$. Under conditions  \eqref{Ezero}, \eqref{Power} and \eqref{uscPower}, there exists a corresponding sequence of Visco-Energetic solutions $(u_n)_n\subset \BV([0,T];\Xs)$ to the rate-independent system $\RIS$, arising from the  viscous corrections $\corrn uv=\tfrac{\mu_n}2 \mdn^2(u,v)$ and  satisfying the initial condition $u_n(0)=u_n^0$.
 \par
 Our first result addresses the behavior of the sequence $(u_n)_n$ in the case $\mu_n\down 0$, under the \emph{sole} conditions   \eqref{Ezero}, \eqref{Power} and \eqref{uscPower} guaranteeing the existence of Visco-Energetic and Energetic solutions, cf.\ Theorems \ref{th:ex-en} and \ref{th:ex-ve}.
\begin{maintheorem}[Convergence to Energetic solutions as $\mu \down 0$]
\label{th:1}
  Let $\calE: [0,T]\times \Xs \to \R$ comply with  \eqref{Ezero}, \eqref{Power} and \eqref{uscPower}. Let  $(u_n^0)_n,\, u_0\subset \Xs$ fulfill \eqref{conv+en-conv_init}
and suppose that $u_0 \in \stabi \mdn 0$. 
Let $(\mu_n)_n \subset (0,\infty)$ be a null sequence, and, correspondingly, let $(u_n)_n\subset \BV([0,T];\Xs)$ be a sequence of   $\VEn$   solutions to  the rate-independent system $\RIS$  fulfilling $u_n(0)=u_n^0$. 
\par
Then, there exist a subsequence  $(u_{n_k})_k$  and  a curve $u\in \BV([0,T];\Xs)$ such that  $u(0)=u_0$,  convergences \eqref{ptw+en-conv} hold, and 
and $u$ is an Energetic solution to $\RIS$.
\end{maintheorem}
\par
We will prove the convergence 
 (along a subsequence) of a sequence of  $\VEn$   solutions, as $\mu_n\uparrow \infty$, to a Balanced Viscosity solution,  
 under the same conditions as in  the existence Theorem \ref{th:ex-bv} for Balanced Viscosity solutions. 
 Hence we need to strengthen \eqref{uscPower}  with  \eqref{uscPower-bis}, and require the chain-rule inequality \eqref{ch-rule-ineq} as well. 
\begin{maintheorem}[Convergence to Balanced Viscosity solutions as $\mu \uparrow \infty$]
\label{th:2}
  Let $\calE: [0,T]\times \Xs \to \R$ comply with  \eqref{Ezero}, \eqref{Power}, \eqref{uscPower-bis}, and  \eqref{ch-rule-ineq}.
  Let  $(u_n^0)_n,\, u_0\subset \Xs$ fulfill 
  \eqref{conv+en-conv_init}.
Let $(\mu_n)_n \subset (0,\infty)$ be a diverging sequence, and, correspondingly, let $(u_n)_n\subset \BV([0,T];\Xs)$ be a sequence of  $\VEn$   solutions to  the rate-independent system $\RIS$   fulfilling $u_n(0)=u_n^0$. 
\par
Then, there exist a subsequence  $(u_{n_k})_k$  and  a curve $u\in \BV([0,T];\Xs)$ such that  $u(0)=u_0$, convergences \eqref{ptw+en-conv} hold,
and $u$ is an Balanced Viscosity solution to $\RIS$.
\end{maintheorem}
Both proofs will be carried out throughout Sections
\ref{s:3} and \ref{s:4}. 
\section{Proofs of Theorems \ref{th:1} and \ref{th:2}}
\label{s:3}
\paragraph{\bf A  preliminary compactness result.}  We start with a Helly-type compactness result for a sequence of  $\VEn$ solutions, associated with parameters $(\mu_n)_n$,  which applies both to the limit $\mu_n\down 0$, and to the limit $\mu_n\up \infty$, under the basic conditions  \eqref{Ezero} and  \eqref{Power} on $\calE$. 
The key starting observation is that, since  
\begin{equation}
\label{key-est-var}
 \Vari {\mdn,\vecostnamep{\mu}}{u}0{t} \geq  \Vari {\mdn}{u}0{t}  \qquad \text{for every } u \in \BV([0,T];\Xs)   \text{ and every } \mu>0, 
 \end{equation}
  every $\VE$ solution complies with the 
upper energy estimate of the energy balance \eqref{global-enbal},  cf.\ \eqref{uee-n} below, where the (either vanishing or blowing up) parameters $\mu_n$  no longer feature. From this energy estimate there stem all the a priori estimates   and compactness properties common to the two singular limits $\mu_n \down 0$ and $\mu_n \up\infty$. 
\begin{proposition}[A priori estimates and compactness]
\label{prop:compactness}
  Let $\calE: [0,T]\times \Xs \to \R$ comply with  \eqref{Ezero} and  \eqref{Power}. Consider a sequence $(u_n)_n\subset \BV([0,T];\Xs)$ of curves starting from initial data $(u_0^n)_n\subset \Xs$ converging to some $u_0\in \Xs$ as in   \eqref{conv+en-conv_init}.
Suppose that the curves $u_n$ fulfill for every $n\in \N$ the upper energy estimate
\begin{equation}
\label{uee-n}
\ene t{u_n(t)} + \Vari {\mdn}{u_n}0{t} \leq \ene 0{u_0^n} +\int_0^t \pw s{u_n(s)} \dd s \quad \text{for all } t \in [0,T]\,.
\end{equation}
Set $V_n: =  V_{u_n}$  (cf.\ \eqref{def-variation-func}).  
\par
Then, 
\begin{equation}
\label{energy-variation-bound}
\exists\, C>0 \ \forall\, n \in \N\, : \qquad \sup_{t\in [0,T]}\perto{u_n(t)} + V_n(T) \leq C\,.
\end{equation}
Furthermore, 
there exist a subsequence $k\mapsto n_k$ and functions $u\in  \BV([0,T];\Xs)$, $\limE,\, \limV\in \BV([0,T])$, and $\limP\in L^\infty(0,T)$, such that
\begin{subequations} 
\label{convs}
\begin{align}
&
\label{convs-a}
u_{n_k}(t) \to u(t) && \text{for all } t \in [0,T],
\\
& 
\label{convs-b}
\ene t{u_{n_k}(t)} \to \limE(t) && \text{for all } t \in (0,T],
\\
& 
\label{convs-c}
V_{n_k}(t) \to \limV(t) && \text{for all } t \in (0,T],
\\
& 
\label{convs-d}
\pw t{u_{n_k}(t)} \weaksto \limP && \text{in }L^\infty(0,T),
\end{align}
\end{subequations}
so that $u(0)=u_0$ and 
there hold
\begin{subequations}
\label{props-lim}
\begin{align}
\label{bv-EST}
&
\md {u(s)}{u(t)} \leq \limV(t) - \limV(s) &&  \text{ for all } 0\leq s\leq t \leq T,
\\
& 
\label{en-EST}
\limE(t) \geq \ene t{u(t)} && \text{ for all } t\in (0,T], \text{ with } \limE(0)= \ene0{u_0}.
\end{align}
\end{subequations}
Furthermore,
for every $t\in \jump u$ there exist two sequences $\alpha_k\up t$ and $\beta_k \down t$ such that
\begin{equation}
\label{diagonal-convs}
u_{n_k}(\alpha_k) \to \lli u t \quad \text{and} \quad  u_{n_k}(\beta_k) \to \rli u t\,.
\end{equation}
Finally,
 the functions $(u,\limE,\limV,\limP)$ comply with 
\begin{equation}
\label{limuee}
\limE(t)+ \limV(t) = \limE(s) +\limV(s)+\int_s^t \limP(r)  \dd r \quad \text{for all } 0\leq s \leq t \leq T\,.
\end{equation}
\end{proposition}
The \emph{proof} follows by trivially adapting the argument for \cite[Thm.\ 7.2]{SavMin16}. 
Let us only mention that estimate \eqref{energy-variation-bound} derives from \eqref{uee-n}, where the integral term on the right-hand side involving the power functional is estimated by resorting to the power control \eqref{Power}. 
 As for \eqref{diagonal-convs}, it can be shown by suitably adapting the Helly-type compactness argument yielding \eqref{convs-a}. 
\par
In the next Secs.\ \ref{ss:3.1} and \ref{ss:4.2}, we will carry out the proof of Theorem \ref{th:1} and, respectively, outline the argument for Theorem \ref{th:2}. In fact,   in  Section \ref{s:4} we will develop the proof of the main technical lower semicontinuity result underlying the limit passage as $\mu_n\up\infty$ in the Visco-Energetic energy balance  ($\mathrm{E}_{\mdn,\vecostnamep{\mu_n}}$)  and  leading to the upper energy estimate \eqref{en-uee-BV}.
\subsection{Proof Theorem \ref{th:1}}
\label{ss:3.1}
We apply Proposition \ref{prop:compactness} and deduce that there exist a subsequence $(u_{n_k})_k $ of   $\VEa{\mu_{n_k}}$   solutions, and a curve $u\in  \BV([0,T];\Xs)$, such that \eqref{convs}, \eqref{props-lim}, and \eqref{limuee} hold.
 In what follows, 
for simplicity we shall denote the sequence of curves $(u_{n_k})_k$
 by $(u_k)_k$  and accordingly write $\mu_k$ in place of $\mu_{n_k}$. 
We split the argument for proving that  the limiting curve  $u$ is an Energetic solution  in some steps.
\par
\emph{\textbf{Claim $1$:} there holds} 
\begin{equation}
\label{tjump}
\begin{cases}
\limE(t) = \ene{t}{u(t)}, \\ \limsup_{k\to\infty} \pw t{u_{k}(t)} \leq \pw t{u(t)} 
\end{cases}
\quad \text{for all } t \in [0,T]\setminus \tjump  \text{ with  } 
\tjump: =\cap_{m\in \N} \cup_{k \geq m} \jump{u_{k}}\,,
\end{equation}
\emph{i.e., the countable set $\tjump$ is  the $\limsup$ of the sets $( \jump{u_{k}})_k$. 
As a result,}
\begin{align}
\label{limp_pw}
&
\limP(t) \leq \pw t{u(t)} \quad \foraa\, t \in (0,T).
\end{align}
To prove \eqref{tjump} at a fixed  $t\in  [0,T]\setminus \tjump$, we observe that, since 
$t\in [0,T]\setminus \jump{u_{k}}$ for every $k\geq m$ and $m \in \N$ a given index (only) depending on $t$, the stability condition
\begin{equation}
\label{stab-k-t}
\ene t{u_{k}(t)} \leq \ene ty + \md {u_{k}(t)}y + \frac{\mu_{k}}2 \mdn^2  (u_{k}(t),y) \qquad  \text{for all } y \in \Xs  \text{ and for all } k \geq m
\end{equation}
holds.
 We choose $y = u(t)$ in \eqref{stab-k-t} and thus deduce that 
$\limsup_{k\to\infty} \ene t{u_{k}(t)} \leq \ene t{u(t)}$. 
   Hence,
 we conclude the energy convergence  
\begin{equation}
\label{en-conv}
 \ene t{u_{k}(t)}\to \ene t{u(t)} \qquad \text{for all } t \in [0,T]\setminus \tjump,
\end{equation}
whence the first of \eqref{tjump}. The $\limsup$ inequality for the power term in \eqref{tjump}   follows from \eqref{uscPower}. 
Then, since  the set $\tjump$ is negligible, we have  for every $t\in (0,T) $ and  $r\in (0,(T{-}t)\wedge t)$
\begin{equation}
\label{Fatou-power}
\int_{t-r}^{t+r} \limP(s) \dd s \leq \limsup_{k\to\infty}  \int_{t-r}^{t+r}  \pw s{u_{k}(s)} \dd s \leq  \int_{t-r}^{t+r}   \pw s{u(s)} \dd s,
\end{equation}
where the second inequality follows from the second of \eqref{tjump} and  the Fatou Lemma, taking into account
that $\sup_{t\in[0,T]}    \pw t{u_{k}(t)}  \leq C_P \sup_{t\in[0,T]}    \pert t{u_{k}(t)}  \leq C$ by virtue of \eqref{Power}, \eqref{prop-pert}, and estimate \eqref{energy-variation-bound}. 
Therefore, \eqref{limp_pw} ensues upon dividing  \eqref{Fatou-power}  by $r$ and  taking the limit as $r\down 0$. 
\par
\emph{\textbf{Claim $2$:} the curve $u$ complies with}
\begin{equation}
\label{uee-st}
\ene t{u(t)} + \Vari {\mdn}{u}s{t} \leq \ene s{u(s)} +\int_s^t \pw r{u(r)} \dd r \quad \text{for all } t \in (0,T],  \,  s\in (0,t)\setminus \tjump, \text{ and  } s=0.
\end{equation}
The upper energy estimate \eqref{uee-st} ensues from \eqref{limuee}, taking into account \eqref{props-lim}, \eqref{tjump}, and \eqref{limp_pw}.
\par
\emph{\textbf{Claim $3$:}}
\begin{equation}
\label{stab-outside-tjump}
u(t)\in \stabi{\mdn}t \quad \text{for every $t\in [0,T]\setminus \tjump$.} 
\end{equation}
It follows from passing to the limit as $k\to\infty$ in the stability condition \eqref{stab-k-t}.
\par
\emph{\textbf{Claim $4$:}}
\begin{equation}
\label{stab-at-lr-lim}
 \lli ut, \, \rli u t \in \stabi{\mdn}t \quad \text{for every $t\in (0,T)$,\  \  $\rli u0 \in \stabi {\mdn}0$, \ \ $\lli u T \in \stabi \mdn T$.} 
 \end{equation}
 Let us only prove the assertion at $t\in (0,T)$ and for $\rli u t$: since the latter right limit exists, we have that
 $
 \rli u t=\lim_{s\down t,\, s \in (t,T) \setminus \tjump} u(s).
$
Therefore, $ \rli u t \in \stabi{\mdn}t$ follows from the previously obtained \eqref{stab-outside-tjump}, combined with the closedness of the  stable set $\stab \mdn$, cf.\ \eqref{4closure-stable-set}.
\par
\emph{\textbf{Claim $5$:}}
\begin{equation}
\label{stab-at-tjump}
u(t) \in \stabi \mdn t \quad \text{for every } t \in (0,T] \cap \tjump.
 \end{equation}
 \emph{Therefore, $u$ complies with the stability condition \eqref{global-stab}.}
 \\
 We consider the upper energy estimate \eqref{uee-st} written on the interval $[s,t]$, for every $s\in (0,t) \setminus \tjump$, and then take the limit of the right-hand side as $s \uparrow t$.  We use that $ \lli u t=\lim_{s\up t,\, s \in (0,t) \setminus \tjump} u(s)$, and that 
 \begin{equation}
 \label{limsup-ene}
\limsup_{s\up t,\, s \in (0,t) \setminus \tjump}  \ene s{u(s)}  \leq \ene t{\lli u t}.
 \end{equation}
 This follows from applying
the stability condition  $u(s) \in \stabi \mdn s$, which holds at all $s \in (0,t) \setminus \tjump$,  with competitor $y= \lli u  t $. Therefore $  \ene s{u(s)}  \leq \ene s{\lli u t} + \md{u(s)}{\lli u t}$, which yields
\begin{equation}
\label{limsup-1}
\limsup_{s\up t,\, s \in (0,t) \setminus \tjump}  \ene s{u(s)}  \leq \limsup_{s\up t,\, s \in (0,t) \setminus \tjump}  \ene s{\lli u t}.
\end{equation}
In turn,
 \begin{equation}
 \label{limsup-2}
 \limsup_{s\up t,\, s \in (0,t) \setminus \tjump} \left(  \ene s{\lli u t} - \ene t {\lli u t} \right) \dd t  \stackrel{(1)}{\leq}    \limsup_{s\up t}
 \int_s^t | \pw r{\lli  u t } | \dd r \stackrel{(2)}{\leq} C   \limsup_{s\up t}
   (t-s)  =0
 \end{equation}
 with {\footnotesize (1)} due to \eqref{2tfc} and {\footnotesize (2)}   to
  the power-control estimate 
   \begin{equation}
 \label{limsup-3}
  | \pw r{\lli  u t } |   \leq C \perto {\lli u t } \ \leq C\,.
  \end{equation}
 In \eqref{limsup-3} the first inequality ensues from
   \eqref{Power} and  \eqref{prop-pert}, while the second one from the lower semicontinuity of $u\mapsto \perto u $, which gives $\perto {\lli u t } \leq \liminf_{s\up t} \perto {u(s)} \leq C$ thanks to      the energy bound 
    $
    \sup_{t\in [0,T]}\perto {u(t)} \leq C,
    $
    deriving from 
  estimate \eqref{energy-variation-bound} by the lower semicontinuity of $\calF_0$. 
       Combining \eqref{limsup-1} with \eqref{limsup-2} we thus conclude  \eqref{limsup-ene}. 
    We also observe that
    \begin{equation}
    \label{liminf-variations}
    \liminf_{s\up t} \Vari{\mdn}{u}{s}t \geq \md{\lli u t }{u(t)}\,.
    \end{equation}
    On account of \eqref{limsup-ene} and \eqref{liminf-variations},
    from \eqref{uee-st} we deduce the jump estimate
    \begin{equation}
    \label{localized-at-jump}
    \ene t{u(t)} + \md{\lli u t }{u(t)} \leq \ene t{\lli u t}  \qquad \text{for every } t \in (0,T] \cap \tjump.
    \end{equation}
    We combine this with the previously obtained stability condition \eqref{stab-at-lr-lim} to conclude \eqref{stab-at-tjump}. 
 \par
\emph{\textbf{Claim $6$:} the curve $u$ complies with the lower energy estimate}
 \begin{equation}
 \label{lee}
 \ene t{u(t)} + \Vari {\mdn}{u}0{t} \geq \ene 0{u(0)} +\int_0^t \pw r{u(r)} \dd r \quad \text{for all } t \in [0,T],
 \end{equation}
\emph{and thus with the energy balance \eqref{global-enbal}.}
\\
We either  apply \cite[Prop.\ 2.1.23]{MieRouBOOK} or \cite[Lemma 6.2, Thm.\ 6.5]{SavMin16}, to conclude \eqref{lee} from the previously obtained \eqref{global-stab}.
 \par
\emph{\textbf{Claim $7$:}  the convergence of the energies $\ene t{u_k(t)} \to \ene t{u(t)}$ holds at every $t\in [0,T]$}.
\par
It follows from \eqref{convs-b} and \eqref{en-EST} that $\liminf_{k\to\infty} \ene t{u_k(t)}  
\geq \ene t{u(t)}$ for every $t\in [0,T]$. To prove the converse inequality for the $\limsup$, we resort to a by now classical argument based on the comparison of the energy balances \eqref{global-enbal} and \eqref{enbal-VE-INTR0}. Indeed, we have 
\[
\begin{aligned}
\limsup_{k\to\infty} \ene t{u_k(t)}   & \stackrel{(1)}{\leq}
\limsup_{k\to\infty} \ene 0{u_k^0} + \limsup_{k\to\infty} \int_0^t \pw r{u_k(r)} \dd  r 
-\liminf_{k\to\infty} \Vari{\mdn,\vecostnamep{\mu_{k}}}{u_{k}}{0}{t}
\\
& 
\stackrel{(2)}{\leq} \ene 0{u_0} + \int_0^t \pw r{u(r)} \dd r 
-  \Vari {\mdn}{u}0{t}\stackrel{(3)}{=} \ene t{u(t)}\,,
\end{aligned}
\]
with {\footnotesize (1)} due to \eqref{enbal-VE-INTR0}, {\footnotesize (2)} following from the assumed convergence of the initial data \eqref{conv+en-conv_init}, from \eqref{convs-d} combined with \eqref{limp_pw}, and from \eqref{key-est-var} and, finally, 
{\footnotesize (3)} due to  the just obtained energy balance  \eqref{global-enbal}.
\par
This concludes the proof of Theorem \ref{th:1}.  
\QED
\subsection{Proof Theorem \ref{th:2}}
\label{ss:4.2}
Proposition \ref{prop:compactness} ensures that any sequence
 $(u_n)_n$ of  $\VE$ solutions, corresponding to parameters $\mu_n\to\infty$, admits a subsequence
  	$(u_{n_k})_k$  converging to a curve   $u\in \BV([0,T];\Xs)$ in the sense of   \eqref{convs} and  \eqref{props-lim};  as in the proof of Thm.\ \ref{th:1},  hereafter 
	we will write $u_k$,  $\mu_k$, and  $\vecostnamep k$ in place of 
	$u_{n_k}$, $\mu_{n_k}$, and 
	$\vecostnamep{\mu_{k}}$, respectively. 
Thanks  to the  chain rule from condition \eqref{ch-rule-ineq},  
 in order to prove  that $u$ is a $\BV$ solution it is  sufficient to verify  the local stability \eqref{loc-stab} and the upper energy estimate \eqref{en-uee-BV}, cf.\ 
\cite[Prop.\ 4.2, Thm.\ 4.3]{MRS12}.
The convergence of the energies  $\ene t{u_k(t)} \to \ene t{u(t)}$ holds at every $t\in [0,T]$ will then follow from comparing the energy balances  \eqref{enbal-VE-INTR0} and 
\eqref{BV-enbal}, similarly as in Claim $7$ of the proof of Thm.\ \ref{th:1}. 
\paragraph{\bf $\vartriangleright$ The local stability condition \eqref{loc-stab}.}
As in the proof of Theorem \ref{th:1}, we introduce the set 
$\tjump := \cap_{m\in \N} \cup_{k\geq m} \jump{u_{k}}$. Since 
 $\cmdn$-stability implies local stability, 
we have that for every $t \in [0,T]\setminus \tjump$ there holds
\begin{equation}
\label{loc-stability-at-k}
\slope \calE t{u_{k}(t)}\leq 1 \quad \text{for all }  k \geq m,
\end{equation}
with $m\in \N$ depending on $t$. 
Taking into  account the energy bound \eqref{energy-variation-bound} as well, we are in a position to exploit the lower semicontinuity property ensured by 
\eqref{uscPower-bis}. Taking the $\liminf_{k\to\infty}$ of \eqref{loc-stability-at-k}, we thus deduce that 
\begin{equation}
\label{loc-stability-at-tjump}
\slope \calE t{u(t)} \leq 1 \quad \text{for all } t \in [0,T]\setminus \tjump.
\end{equation}
We also conclude that 
\begin{equation}
\label{loc-stability-at-left-right-lims}
\slope \calE t{\lli u t},\ \slope \calE t{\rli u t}\leq 1 \quad \text{for all } t \in (0,T),
\end{equation}
and analogously for $\slope \calE 0{\rli u0}$ and $\slope \calE T{\lli uT}$, by arguing in the very same way as for \emph{\bf Claim $4$} in the proof of  Theorem \ref{th:2}. 
Clearly, we  then have the local stability condition at all points in $[0,T]\setminus \jump u$.
\\
\paragraph{\bf $\vartriangleright$ The upper energy estimate \eqref{en-uee-BV}.}  Combining the energy bound \eqref{energy-variation-bound} and
the slope estimate \eqref{loc-stability-at-k}
with convergence \eqref{convs-a} and resorting to \eqref{uscPower-bis}, we conclude that
$
\limsup_{k\to\infty} \pw t{u_{k}(t)} \leq \pw t{u(t)}
$ for all $t\in [0,T]\setminus \tjump$. Therefore, the very same argument as for  \emph{\bf Claim $1$}  in the proof of  Theorem \ref{th:2} yields that $\limP(t) \leq \pw t{u(t)}$ for almost all $t\in (0,T)$. 
All in all, taking the $\liminf_{k\to\infty}$ in  ($\mathrm{E}_{\mdn,\vecostnamep{\mu_{k}}}$)  and exploiting the initial data convergence \eqref{conv+en-conv_init}, the previously obtained  \eqref{en-EST}, and the above estimate for $\limP$,  we infer that 
\[
\ene T{u(T)} +  \liminf_{k\to\infty} \Vari{\mdn,\vecostnamep{\mu_{k}}}{u_{k}}{0}{T}\leq \ene 0{u(0)} + \int_0^T  \pw r{u(r)} \dd r\,. 
\]
In order to conclude \eqref{en-uee-BV}, 
it  thus remains to show 
that 
\[
 \liminf_{k\to\infty} \Vari{\mdn,\vecostnamep{\mu_{k}}}{u_{k}}{0}{T} \geq \Vari{\mdn,\bvcostname}{u}{0}{T}\,. 
\]
This will be guaranteed by the upcoming result, whose proof will be developed throughout Section \ref{s:4}.
\QED
\begin{theorem}
	\label{th:lsc}
Let $\calE: [0,T] \times \Xs \to \R$ comply with \eqref{Ezero},  \eqref{Power}, and  \eqref{uscPower-bis}. Let
	 $\mu_k \uparrow \infty$ and  $(u_k)_k,\, u \in \BV([0,T];\Xs)$ fulfill
	\begin{subequations}
	\label{conditions-seqs}
	\begin{align}
	&
	\label{conditions-seqs-est}
	  \exists\, C_F>0 \ \forall\, k \in \N\, : \ 
 \sup_{t\in [0,T]}\perto{u_{k}(t)} \leq C_F\,,
	 \\
	 &
	 \label{conditions-seqs-conv}
	u_k(t) \to u(t)  \quad \text{for every $t\in [0,T]$,}
	\\
	&
	\label{conditions-diag}
	 \forall\, t \in \jump u \ \exists\, (\alpha_k)_k,\, (\beta_k)_k \subset [0,T] \text{ with } \alpha_k\up t, \ \beta_k \down t \text{ and } 
	u_k(\alpha_k) \to \lli u t, \ u_k(\beta_k) \to \rli u t.
	\end{align}
			\end{subequations}
		Then,
		\begin{equation}
		\label{desired-lsc}
		\liminf_{k\to\infty} \Vari{\mdn,\vecostnamep k}{u_{k}}{a}{b} \geq \Vari{\mdn,\bvcostname}{u}{a}{b} \qquad 
		\text{for all } [a,b]\subset [0,T].
		\end{equation}	
\end{theorem}
\section{Proof of Theorem \ref{th:lsc}}
\label{s:4}
Let us mention in advance the argument for proving the lower semicontinuity inequality \eqref{desired-lsc} follows the same steps, outlined below,  as those for the lower semicontinuity result \cite[Prop.\ 7.3]{MRS13}
in the context of the limit passage from `viscous' gradient systems to $\BV$ solutions. 
Nevertheless, we have to cope with  the (nontrivial)  technical issues peculiar of the fact that the kind of transitions describing the system behavior at jumps changes upon passing from $\VE$ to $\BV$ solutions. This problem will be addressed in the proof of Proposition  \ref{prop:real-tech-diff}  ahead. 
\paragraph{\bf Outline of the proof of Theorem   \ref{th:lsc}.} 
Up to the extraction of a (not relabeled) subsequence and modifying the constant $C_F$ from \eqref{conditions-seqs-est}, we may suppose that 
\begin{equation}
\label{sup-variazioni}
 \sup_k \Vari{\mdn,\vecostnamep k}{u_{k}}{a}{b}\leq C_F, 
\end{equation} 
too.
We introduce a sequence of  non-negative and bounded Borel measures $\eta_k$ by defining them on intervals
 via 
\[ 
\eta_k ([a,b]): =   \Vari{\mdn,\vecostnamep k}{u_{k}}{a}{b}  \quad \text{for all } [a,b]\subset [0,T].
\]
In view of \eqref{sup-variazioni}, we have that, up to a further extraction, there exists a Borel measure $\eta$ such that $\eta_k\weaksto \eta$ in duality with $\rmC([0,T])$. 
Observe that, by \eqref{key-est-var},  we have 
\[
\eta([a,b])\geq \limsup_{k\to\infty}\eta_k ([a,b]) \geq \limsup_{k\to\infty} \Vari{\mdn}{u_{k}}{a}{b} \geq \Vari{\mdn}{u}{a}{b} \geq \mum u{\mathrm{d}} ([a,b]),
\]
with $ \mum u{\mathrm{d}} $ the diffuse measure associated with $u$ via \eqref{decompose}. Therefore we obtain 
\begin{equation}
\label{step1}
\eta \geq  \mum u{\mathrm{d}}\,.
\end{equation}
\par
We now exploit Proposition \ref{prop:real-tech-diff} ahead to conclude that, for every $t\in \jump u$ and any two sequences 
$\alpha_k\up t $ and  $\beta_k \down t$ fulfilling \eqref{conditions-diag},
there holds
\begin{equation}
\label{step2}
\eta(\{t\}) \geq \limsup_{k\to\infty}\eta_k ([\alpha_k,\beta_k]) \geq \liminf_{k\to\infty}\eta_k ([\alpha_k,\beta_k]) \geq \bvcost t{\lli u t }{\rli u t}\,.
\end{equation}
Analogously, we can prove that
\begin{equation}
\label{step3-bis}
\limsup_{k\to\infty}\eta_k ([\alpha_k,t]) \geq \bvcost t{\lli u t }{u(t)}, \qquad  \limsup_{k\to\infty}\eta_k ([t,\beta_k]) \geq \bvcost t{u (t) }{\rli u t}\,.
\end{equation}
\par
Arguing in the very same way as in the proof of  \cite[Prop.\ 7.3]{MRS13}, we combine  \eqref{step1}, \eqref{step2}, and \eqref{step3-bis}
with the representation 
\[
\begin{aligned}
&
\Vari {\mdn,\bvcostname}u{a}{b}\\
 &  =  \mum u{\mathrm{d}}([a,b]) +  \Jvar {{\bvcostname}}{u}{a}{b}
\\
& 
=  
  \mum u{\mathrm{d}}([a,b]) +
 \bvcost{a}{u(a)}{\rli u{a}} +  \bvcost{b}{\lli u{b}}{ u(b)} + \sum_{t\in \jump u \cap (a,b)} \left(
 \bvcost{t}{\lli u{t}}{u(t)} {+}  \bvcost{t}{u(t)}{\rli u{t}}\right), 
 \end{aligned}
\]
cf.\ \eqref{serve?}, to 
conclude the desired  lower semicontinuity inequality \eqref{desired-lsc}.
\QED
The proof of the upcoming result is  developed throughout  Section \ref{ss:4.1}.
\begin{proposition}
	\label{prop:real-tech-diff}
Let $\calE: [0,T] \times \Xs \to \R$ comply with \eqref{Ezero},  \eqref{Power}, and  \eqref{uscPower-bis}. Let
	 $\mu_k \uparrow \infty$ and  $(u_k)_k,\, u \in \BV([0,T];\Xs)$ fulfill
	\eqref{conditions-seqs} and \eqref{sup-variazioni}. For every $t\in \jump u$, pick two sequences $(\alpha_k)_k,\, (\beta_k)_k$ converging to $t$ and fulfilling \eqref{conditions-diag}. 
	Then,
	\begin{equation}
\label{lsc-dura}
 \liminf_{k\to\infty}   \Vari{\mdn,\vecostnamep k}{u_{k}}{\alpha_k}{\beta_k} \geq \bvcost t{\lli u t }{\rli u t}\,.
\end{equation}
\end{proposition}

\subsection{Proof of Proposition  \ref{prop:real-tech-diff}}
\label{ss:4.1}
We split the argument in some steps, some of which in turn rely on some technical results proved in the Appendix.
\paragraph{\bf Step $1$: reparameterization.}
The curve $u_k$ has at most countably many jump points $(t_m^k)_{m\in M_k}$ between the points $\alpha_k$ and $\beta_k$.
 We now suitably reparameterize both the continuous pieces of the trajectory $u_k$, as well as the optimal transitions 
$\teta_j^k$ connecting the left and right limits $\lli {u_k}{t_j^k}$ and  $\rli {u_k}{t_j^k}$ at a jump point $t_j^k$. We will then glue all of them together to obtain a sequence of curves $(\invcur k)_k$, defined on  compact sets $(C_k)_k$,  which shall enjoy suitable estimates (cf.\ Step $2$), allowing for a refined compactness argument both for the curves $\invcur k $ and for the sets $C_k$. 
\par
We set 
\[
\newmresc k: = \beta_k -\alpha_k +  \Vari{\mdn,\vecostnamep k}{u_{k}}{\alpha_k}{\beta_k}+ \sum_{m \in M_k} 2^{-m}\,
\]
and define the  rescaling function $\nresc k : [\alpha_k, \beta_k]\to [0,\newmresc k]$   by 
 \[
 \nresc k(t): 
 = t-\alpha_k +    \Vari{\mdn,\vecostnamep k}{u_{k}}{\alpha_k}{t}+ \sum_{\{m \in M_k :\, t_m^k \leq t\}} 2^{-m}\,.
 \]
 Observe that $\nresc k $ is strictly increasing, with jump set $\jump{\nresc k} = (t_m^k)_{m\in M_k}$. We  introduce the notation
 \[
 I_m^k: =  (\nresc k(t_m^k-), \nresc k(t_m^k+)),
  \quad I_k: = \cup_{m\in M_k}  I_m^k, \quad  \Lambda_k: =[ \nresc k (\alpha_k), \nresc k (\beta_k)]. 
 \]
 On   $\Lambda_k\setminus I_k$ 
    the inverse $\ninvresc k: \Lambda_k \setminus I_k \to [\alpha_k, \beta_k]$ 
    of $\nresc k $ 
    is well defined and Lipschitz continuous. We set
    \begin{equation}
    \label{invcurk-1}
    \invcur k (s): = (u_k \circ \ninvresc k)(s) \qquad \text{for all } s \in \Lambda_k \setminus I_k\,.
    \end{equation}
    The curve $\invcur k$ is also Lipschitz, and satisfies
    \begin{equation}
    \label{est-bv-invcur}
    \begin{aligned}
    &
     \Vari{\mdn,\vecostnamep k}{\invcur k}{s_0}{s_1} \leq (s_1{-}s_0) \quad \text{for all } [s_0,s_1]\subset \Lambda_k \setminus I_k\,.
     \end{aligned}
    \end{equation}
     We  check \eqref{est-bv-invcur} in the case in which $s_0= \nresc k(t_0)$ and $s_1 = \nresc k(t_1)$, with $t_0 <t_1$ belonging to the same connected component of $[\alpha_k,\beta_k]\setminus (t_m^k)_{m\in M_k}$ (the other case is completely analogous).
    Then, we observe that
    \[
    s_1-s_0 = \nresc k(t_1)-\nresc k(t_0) = t_1-t_0 +  \Vari{\mdn,\vecostnamep k}{u_{k}}{t_0}{t_1}
    \geq  \Vari{\mdn,\vecostnamep k}{\invcur k}{s_0}{s_1}\,. 
    \] 
    \par 
We now  recall 
\cite[Thm.\ 3.14]{SavMin16}, ensuring that at every jump point $t_m^k$ there exists an optimal transition $\teta_m^k$ that is continuous on a compact set $E_m^k$, 
\emph{tight} (i.e.\ it fulfills $\teta_m^k(J^-) \neq \teta_m^k(J^+)$ for every  ``hole'' $J\in \hole  {E_m^k}$), and 
such that 
\begin{equation}
\label{tight-OJT}
\begin{aligned}
&
\lli u{t_m^k} = \teta_m^k ((E_m^k)^-), \qquad \rli u{t_m^k} = \teta_m^k ((E_m^k)^+), \qquad u(t_m^k)\in \teta_m^k(E_m^k), \\
&
\begin{aligned}
\ene {t_m^k}{\lli u{t_m^k}} - \ene {t_m^k}{\rli u{t_m^k}}   & = \vecost{t_m^k}{\lli u{t_m^k}}{\rli u{t_m^k}}  = \tcost{\VE}{t_m^k}{\teta_m^k}{E_m^k}
\\ & 
=  \Vars {\mdn}{\teta_m^k} {E_m^k} + \Gap{\mdn}{\teta_m^k} {E_m^k} + \sum_{r\in E_m^k{\setminus}(E_m^k)^+}  \rstab {t_m^k} {\teta_m^k(r)}\,. 
\end{aligned}
 \end{aligned}
\end{equation}
We adapt the calculations from \cite[Lemma 5.1]{SavMin16} 
and define the rescaling function $\rresc mk $ on $ E_m^k $ 
 by
\[
\begin{aligned}
 \rresc mk (t): = &  \frac1{2^m} \frac{t-(E_m^k)^-}{(E_m^k)^+ - (E_m^k)^-} +  \Vars {\mdn}{\teta_m^k} {E_m^k \cap [(E_m^k)^-,t]} 
 \\ & 
 + \Gap{\mdn}{\teta_m^k} {E_m^k \cap [(E_m^k)^-,t]}  +\sum_{r\in  [(E_m^k)^-,t]\setminus (E_m^k)^+}  \rstab {t_m^k}{\teta_m^k(r)}+\nresc k(t_m^k-)
 \end{aligned}
\]
for all $ t \in E_m^k$.
It can be checked that $\rresc mk$ is continuous and strictly increasing,
with  image a compact set  $S_m^k\subset I_m^k $ such that 
\[
\begin{aligned}
&
(S_m^k)^- =  \rresc mk ((E_m^k)^-) =\nresc k(t_m^k-) \quad \text{ and }
\\
& 
\begin{aligned}
(S_m^k)^+ =   \rresc mk ((E_m^k)^+)  & =  \frac1{2^m} + \Vars {\mdn}{\teta_m^k} {E_m^k} + \Gap{\mdn}{\teta_m^k} {E_m^k}  +\sum_{r\in  E_m^k\setminus (E_m^k)^+}  \rstab {t_m^k}{\teta_m^k(r)}  +\nresc k(t_m^k-) \\ & = \nresc k(t_m^k+)\,.
\end{aligned}
 \end{aligned}
\]
The inverse function   $\siresc mk :  S_m^k  \to E_m^k$ is Lipschitz continuous.
\par
We then introduce the set
\[
C_k: = (\Lambda_k {\setminus} I_k) \cup (\cup_{m\in M_k} S_m^k)\,.
\]
It is not difficult to check that $C_k$ is a closed subset of $\Lambda_k$. We 
 extend  the functions $\ninvresc k$ and 
 $\invcur k $,  so far defined on $\Lambda_k \setminus I_k$, only, to the set $C_k$ 
by setting
\[
\ninvresc k(s) \equiv t_m^k \quad\text{and} \quad
\invcur k (s) : = \teta_m^k( \siresc mk(s)) \qquad \text{whenever } s \in S_m^k \text{ for some  } m \in M_k. 
\] 
Since $  \lli u{t_m^k} = \teta_m^k ((E_m^k)^-) $ and  $ \rli u{t_m^k} = \teta_m^k ((E_m^k)^+)$, we have that the extended curve 
 $\invcur k \in \mathrm{C}(C_k;\Xs) $. Furthermore, $\invcur k \in  \BV(C_k;\Xs)$: 
 indeed, 
\begin{subequations}
\label{possibly-useful}
\begin{equation}
\label{possibly-useful-a}
\begin{gathered}
\Vars {\mdn}{\invcur k} {S_m^k}=
 \Vars {\mdn}{\teta_m^k} {E_m^k} , \qquad  \Gap{\mdn}{\invcur k} {S_m^k} =\Gap{\mdn}{\teta_m^k} {E_m^k}, 
 \\
 \sum_{s\in S_m^k{\setminus}\{(S_m^k)^+\}}  \rstab{t_m^k}{\invcur k(s)} =  \sum_{r\in E_m^k{\setminus}\{(E_m^k)^+\}} \rstab{t_m^k}{\teta_m^k(r)}, 
 \end{gathered}
 \end{equation}
 as well as 
 \begin{equation}
 \label{possibly-useful-b}
  \Vars {\mdn}{\invcur k} {S_m^k \cap [s_0,s_1]} \leq  (s_1-s_0) \qquad \text{for all } s_0,s_1\in  S_m^k \text{ with } s_0<s_1.
\end{equation}
\end{subequations} 
\paragraph{\bf Step $2$:  a priori estimates}
  It follows from \eqref{sup-variazioni} and from the fact that 
 $(\beta_k {-} \alpha_k) \down 0$, 
  that 
  \begin{equation}
\label{quoted-later}C_k^+ = \newmresc k \leq  \beta_k - \alpha_k+ \Vari{\mdn,\vecostname}{u_k}{\alpha_k}{\beta_k} + 2 \leq  2 C_F 
\end{equation}
(up to modifying the constant $C_F$). 
Moreover,  in view of \eqref{sup-variazioni},  \eqref{est-bv-invcur},  and \eqref{possibly-useful-b} we have 
  \begin{subequations}
  \label{estimates-rescaled-uk}
\begin{align}
&
\label{bound-variations-k}
\sup_{k\in \N} \Vars{\mdn}{\invcur k}{C_k} \leq C, 
\\
&
\label{1-Lip}
\Vars{\mdn}{\invcur k}{C_k \cap [s_0,s_1]} \leq (s_1{-}s_0) \quad \text{for all } s_0,s_1 \in C_k \text{ with } s_0<s_1 \text{ and all } k \in \N.
\end{align}
 Finally, we remark that 
\begin{equation}
\label{bound-energies-k} 
\sup_{k\in \N} \sup_{s\in C_k} \perto {\invcur k(s)} \leq C_F.
\end{equation}
\end{subequations}
Indeed, we have that 
\[
 \sup_{s \in \Lambda_k {\setminus} I_k} \perto {\invcur k(s)}  = \sup_{t\in [\alpha_k,\beta_k]{\setminus} (t_m^k)_{m\in M_k}} \perto {u_k(t)} \leq C_F
\]
in view of  \eqref{conditions-seqs}. Furthermore, it follows from \cite[Thm.\ 3.16]{SavMin16} that for all $r\in E_m^k$ there holds
\[
\begin{aligned}
\ene{t_m^k}{{\teta}_{m}^k(r)} + \md {{\teta}_{m}^k(r)}{{\teta}_{m}^k((E_m^k)^-)} \leq 
\ene{t_m^k}{{\teta}_{m}^k(r)} + \Vars {\mdn}{\teta_m^k}{E_m^k \cap [(E_m^k)^-,r]}  & \leq \ene{t_m^k}{{\teta}_{m}^k((E_m^k)^-)}
\\ & 
=\ene{t_m^k}{\lli {u_k}{t_m^k}}\,.
\end{aligned}
\]
Therefore, 
\[
 \sup_{s\in S_m^k }\perto{ \invcur k(s)} =  \sup_{r\in E_m^k} \perto{{\teta}_{m}^k(r)} \leq \perto{\lli {u_k}{t_m^k}} \leq C_F\,.
\]
All in all, we conclude \eqref{bound-energies-k}. 
\paragraph{\bf Step $3$: compactness} By virtue of estimates \eqref{estimates-rescaled-uk}, we are in a position to apply the 
compactness result  \cite[Thm.\ 5.4]{SavMin16} and conclude that 
 there exist a (not relabeled) subsequence,  a compact set   $C\subset [0,2C_F]$,
  and a function $\invcu \in \BV(C;\Xs)$ such that, as $k\to\infty$, there hold
\begin{enumerate}
\item $C_{k} \to C$ \`a la Kuratowski;
\item $\mathrm{graph}(\invcu)\subset \mathrm{Li}_{k\to\infty} \mathrm{graph}(\invcur {k})$;
\item whenever $(s_{k})_k  \in C_{k}$ converge to $s\in C$, then $\invcur {k}(s_{k})\to \invcu(s)$;
\item 
$\invcur {k}((C_{k})^\pm )\to \invcu(C^\pm)$.
\end{enumerate}
Therefore, $\invcu(C^-) =  \lli u t$, and $\invcu(C^+) = \rli u t$. Furthermore, it follows from \eqref{1-Lip} that the curve   $\invcu $   is Lipschitz on $C$.
 Finally, for later use  let us point out that, since 
  the functions $\ninvresc k$ take values in the intervals $[\alpha_k,\beta_k]$ shrinking to the singleton $\{t\}$, there holds 
\begin{equation}
\label{singleton}
\lim_{k\to\infty}  \sup_{s\in C_k} | \ninvresc k(s)- t | =0. 
\end{equation} 
 \paragraph{\bf Step $4$: connectedness of $C$}
Observe that, since the sets $C_k$ are not, in general, connected,
 we cannot immediately deduce that $C$ is connected. 
 We will however show that,
\begin{equation}
\label{to-holes}
\forall\, I \in \hole C \text{ there holds } \invcu (I^-) = \invcu(I^+) =: \invcu_I.
\end{equation} 
In view of this, we may extend $\invcu$ to the whole interval $[0,C^+]$ by defining
\[
\invcu(s): = \invcu_I \qquad \text{for all } s \in I \quad \text{for all } I \in \hole C.
\]
 Hereafter, we will replace $C$ by $[0,C^+]$. 
We will split the proof of \eqref{to-holes} in two claims.
\par
\emph{\textbf{Claim $1$:} 
 for every $ I \in \hole C $ there exist $J_k $ such that}
\begin{equation}
\label{che-fatica}
J_k \in \hole {C_k} \text{ and } \lim_{k\to\infty} J_k^-=I^-, \quad \lim_{k\to\infty} J_k^+=I^+.
\end{equation}
This follows by repeating the very same arguments as in the proof of  \cite[Thm.\ 5.3]{SavMin16}. 
\par
\emph{\textbf{Claim $2$:} there holds $\invcu (I^-) = \invcu(I^+) $.
} 
In view of the compactness property (3) from Step $3$, 
there holds   $\invcur k (J_k^\pm) \to \invcu (I^\pm)$. Therefore, 
\[
\begin{aligned}
\md{\invcu (I^-)}{\invcu(I^+)} = \lim_{k\to\infty} \md{\invcur k (J_k^-)}{\invcur k (J_k^+)}  & \leq 
\limsup_{k\to\infty}
 \frac1{2\mu_k^{1/2}}  \left( \mu_k \mdn^2(\invcur k (J_k^-),\invcur k (J_k^+)) + 1\right) \\ &  \leq
\limsup_{k\to\infty} \frac1{\mu_k^{1/2}}   \left( \Vari{\mdn,\vecostname_k}{u_k}{\alpha_k}{\beta_k} +1 \right)  =0, 
\end{aligned}
\]
where we have used Young's equality and  estimate  \eqref{sup-variazioni}. 
  \paragraph{\bf Step $5$: estimate of the transition cost and conclusion of the proof}  With Steps $3$ and $4$ we have shown that the Lipschitz continuous  curve $\invcu$ is defined on the interval $[0,C^+]$  and connects the left and right limits $\lli u t $ and $\rli u t$. 
  We now aim to prove that 
  \begin{equation}
  \label{what-we-want}
   \liminf_{k\to\infty} \Vari{\mdn,\vecostname_k}{u_{k}}{\alpha_k}{\beta_k}   \geq 
  \tcost{\BV}{t}{\invcu}{[0,C^+]}   \geq \bvcost t{\lli u t }{\rli u t},
      \end{equation}
      which will lead to \eqref{lsc-dura}.
  \par
 Indeed, it follows from Lemma \ref{l:absolute-continuity} that 
  \begin{equation}
\label{formula-integrale}
\begin{aligned}
& \tcost{\BV}{t}{\invcu}{[0,C^+]}   = \int_0^{C^+} |\invcu'|(s) \left( \slope \calE t{\invcu(s)} \vee 1 \right)  \dd s    
\\ &  \qquad
=\sup\left\{ \sum_{i=1}^N \md{\invcu(\sigma_{i-1})}{\invcu(\sigma_i)}  \inf_{\sigma \in [\sigma_{i-1},\sigma_i]}\left( \slope \calE t{\invcu(\sigma)} \vee 1 \right)   \, : \ (\sigma_i)_{i=1}^N \in \mathfrak{P}_f([0,C^+]) \right\}\,.
\end{aligned}
\end{equation}
Therefore, in what follows we will prove that
  \begin{equation}
  \label{we-need-this}
   \liminf_{k\to\infty} \Vari{\mdn,\vecostname_k}{u_{k}}{\alpha_k}{\beta_k} \geq   
   \sum_{i=1}^N \md{\invcu(\sigma_{i-1})}{\invcu(\sigma_i)}  \inf_{\sigma \in [\sigma_{i-1},\sigma_i]}\left( \slope \calE t{\invcu(\sigma)} \vee 1 \right) 
  \end{equation}
   for every  $(\sigma_i)_{i=1}^N \in \mathfrak{P}_f([0,C^+])$.
  \par
  Let us consider a given partition  $ (\sigma_i)_{i=1}^N  \in \mathfrak{P}_f([0,C^+])$ and fix an index ${j} \in \{1,\ldots,N\}$. 
   Preliminarily, we observe that, by the compactness property (1)  in Step $3$,
there exist sequences $(\sigma_{j-1}^k)_k$, 
$(\sigma_j^k)_k \subset C_k$ such that 
\begin{equation}
\label{approximation-partition}
\sigma_{j-1}^k\to\sigma_{j-1}, \  
\sigma_j^k\to\sigma_j \quad \text{and} \quad 
 \invcur k (\sigma_{j-1}^k)\to\invcu(\sigma_{j-1}),\  
\invcur k (\sigma_j^k)\to\invcu(\sigma_j)\quad \text{as }k \to\infty,
\end{equation}
where the second convergence  follows from   the compactness property (3).
  We  now distinguish two cases
  \begin{enumerate}
  \item $  \inf_{\sigma \in [\sigma_{j-1},\sigma_j]}\left( \slope \calE t{\invcu(\sigma)} \vee 1 \right)= 1 $;
  \item   $   \inf_{\sigma \in [\sigma_{j-1},\sigma_j]}\slope \calE t{\invcu(\sigma)}   > 1 $. 
    \end{enumerate} 
    Clearly, the second case is equivalent to $\inf_{\sigma \in [\sigma_{j-1},\sigma_j]}\left( \slope \calE t{\invcu(\sigma)} \vee 1 \right)  > 1$.

  \noindent
\textbf{Case (1):} 
  In view  of \eqref{approximation-partition}, we have
\begin{equation}
\label{ingredient-step-1}
 \md{\invcu(\sigma_{j-1})}{\invcu(\sigma_j)}  \inf_{\sigma \in [\sigma_{j-1},\sigma_j]}\left( \slope \calE t{\invcu(\sigma)} \vee 1 \right)  =   \lim_{k\to\infty} \md{\invcur k(\sigma_{j-1}^k)}{\invcur k(\sigma_j^k)}\,.
\end{equation}
\textbf{Case (2):}    We have that $\slope \calE t{\invcu(\sigma)}> \delta>1$ for all 
$\sigma \in [\sigma_{j-1},\sigma_j]$.  First of all, we observe that 
\begin{equation}
\label{claim-one}
 \exists\, \bar{\delta} \in (1,\delta) \quad 
\exists\, \bar{k}\in \N 
\quad  \inf_{k \geq \bar{k}}\,  \inf_{\sigma \in  [\sigma_{j-1}^k,\sigma_j^k] {\cap} C_k}
\slope \calE {\ninvresc k (\sigma)}{\invcur k(\sigma)} \geq \bar\delta\,. 
\end{equation}
To show this, we argue by contradiction and suppose that there exists a (not relabeled) subsequence  along which $ \inf_{\sigma \in  [\sigma_{j-1}^k,\sigma_j^k] \cap C_k}
\slope \calE {\ninvresc k (\sigma)}{\invcur k(\sigma)} \leq1$. Since for every $k\in \N$ the $\inf$ 
 on the compact set $[\sigma_{j-1}^k,\sigma_j^k] \cap C_k$  
is attained by lower semicontinuity of the map $\sigma \mapsto \slope \calE {\ninvresc k (\sigma)}{\invcur k(\sigma)}$, we deduce that there exists a sequence $(\tilde{\sigma}_k )_k$ with $\slope \calE {\ninvresc k (\tilde{\sigma}_k)}{\invcur k(\tilde{\sigma}_k)} \leq1$, converging up to a subsequence to some $\tilde \sigma \in  [\sigma_{j-1},\sigma_j]$. 
Now,  $\ninvresc k (\tilde{\sigma}_k) \to t$ by \eqref{singleton} and $\invcur k (\tilde{\sigma}_k) \to \invcu(\tilde\sigma)$ by the
compactness property (3) from Step $3$. Hence,  using the lower semicontinuity of $|\mathrm{D}\calE|$ granted by \eqref{uscPower-bis} we conclude that 
$\slope \calE {t}{\invcu(\tilde\sigma)} \leq1$, in contradiction with the assumption that
 $\inf_{\sigma \in [\sigma_{j-1},\sigma_j]}\slope \calE t{\invcu(\sigma)}  > 1 $. 
\par
 Observe that \eqref{claim-one} implies that $\rstab{\ninvresc k (\sigma)}{\invcur k(\sigma)}>0 $
for all $ \sigma \in  [\sigma_{j-1}^k,\sigma_j^k] \cap C_k $ and all $ k \geq \bar k. $
We now  deduce    the \emph{uniform} positivity property
\begin{equation}
\label{stronger-form}
\exists\, r>0 \qquad \inf_{k \geq \bar{k}}\, \inf_{\sigma \in  [\sigma_{j-1}^k,\sigma_j^k] {\cap} C_k} \rstab{\ninvresc k (\sigma)}{\invcur k(\sigma)} \geq r\,.
\end{equation}
Indeed,    as for \eqref{claim-one}
we proceed by contradiction: if \eqref{stronger-form} did not hold, there would exist a 
 sequence
  $(\tilde{\sigma}_k )_k$  with $ \rstab{\ninvresc k (\tilde{\sigma}_k)}{\invcur k(\tilde{\sigma}_k)} \to 0$, converging 
  to some $\tilde \sigma \in  [\sigma_{j-1},\sigma_j]$ that would fulfill $\rstab{t}{\invcu (\tilde\sigma)}=0$ by the lower semicontinuity of $\mathcal{R}$.  Now, 
   by property  \eqref{propR},  $\rstab{t}{\invcu (\tilde\sigma)}=0$  
  would  imply that $(t,\invcu(\tilde\sigma)) $
 belongs to the stable set $\mathscr{S}_{\mathsf{D}}$. In turn,  the $\mathsf{D}$-stability condition \eqref{cmdn-stab}  would imply that $\slope \calE {t}{\invcu (\tilde\sigma)}\leq 1$, against the standing assumption that  $\inf_{\sigma \in [\sigma_{j-1},\sigma_j]}\slope \calE t{\invcu(\sigma)}  > 1 $.
\par
Now, \eqref{stronger-form} 
 entails that $\ninvresc k (\sigma) \in (t_m^k)_{m\in M_k}$
for all $ \sigma \in  [\sigma_{j-1}^k,\sigma_j^k] \cap C_k =:  \mathscr{L}_k$. But then,
it is not difficult to realize that the function  $\ninvresc k$ must be constant on 
$ \mathscr{L}_k$. Namely, there exists ${m}_k \in M_k$ such that 
$\ninvresc k (\sigma) \equiv t_{{m}_k}^k$ for all $\sigma \in  \mathscr{L}_k$. 
It was observed in 
\cite[Rmk.\ 3.15]{SavMin16} that the set
$
C_{k}^{\calR}: = \{ s \in S_{{m}_k}^k\setminus \{(S_{{m}_k}^k)^+\}\, : \rstab {t_{{m}_k}^k}{\invcur k(s)}>0\}  $  is discrete. 
Trivially adapting the  argument from \cite[Rmk.\ 3.15]{SavMin16},  from \eqref{stronger-form} we  in fact conclude that 
 for all $k\geq\bar{k}$ the set  $  \mathscr{L}_k \subset C_{k}^{\calR}$  consists of finitely many points 
$ (r_{\ell}^k)_{\ell=1}^{L_k}$, and 
 that the cardinality $L_k $ of the sets  $\mathscr{L}_k$ is uniformly bounded with respect to $k$, i.e.
\begin{equation}
\label{bound-cardinalita}
\sup_{k \geq \bar{k}} L_k \leq C<\infty.
\end{equation}
Furthermore, notice that 
 $r_\ell^k$  is the extremum of a hole of $C_k$ for every $\ell =1, \ldots, L_k$. 
\par\noindent The compactness statement from Step $3$ (cf.\ again \cite[Thm.\ 5.4]{SavMin16}) applies, yielding that, up to a subsequence,
 \begin{enumerate}
\item
 the sets $   ({\mathscr{L}}_k)_k $ converge in the sense of Kuratowski to a 
  finite, thanks to \eqref{bound-cardinalita}, 
 set $\mathscr{L} = (r_l)_{l=1}^L  \subset [\sigma_{j-1},\sigma_j]$,  such that 
 $\sigma_{j-1},\, \sigma_j \in \mathscr{L}$. 
 \item
 for every $r_l \in \mathscr{L}$ there exists a sequence $
 (r_{\ell_k}^k(l))_k$, with $r_{\ell_k}^k(l) \in {\mathscr{L}}_k $ for every $k\in \N$, such that 
 $\invcur k (r_{\ell_k}^k(l)) \to \invcu (r_l)$. From now on, we will use the simplified notation $r_k(l)$ in place of $r_{\ell_k}^k(l)$;
 \item
  whenever $  r_{{\ell_n}}^{k_n} \in \overline{\mathscr{L}}_{k_n}$ converge to some $r_{l} \in \mathscr{L}$ as $n\to\infty$, then 
$\invcur {k_n} ( r_{{\ell_n}}^{k_n}) \to \invcu (r_l)$. 
\end{enumerate}
\par
We now estimate $\md{\invcu(\sigma_{j-1})}{\invcu(\sigma_j)} \inf_{\sigma \in [\sigma_{j-1},\sigma_j]} \left( \slope \calE t{\invcu(\sigma)} \vee 1 \right) $ by interpolating between the points $\sigma_{j-1}$ and $\sigma_j$ the points   $ \mathscr{L}=(r_l)_{l=1}^L$. 
Thus we have 
\begin{equation}
\label{case2.2}
\begin{aligned}
&
\md{\invcu(\sigma_{j-1})}{\invcu(\sigma_j)} \inf_{\sigma \in [\sigma_{j-1},\sigma_j]} \left( \slope \calE t{\invcu(\sigma)} \vee 1 \right)  
\\
 & \leq
\md{\invcu(\sigma_{j-1})}{\invcu(\sigma_j)}  + \md{\invcu(\sigma_{j-1})}{\invcu(\sigma_j)}  \inf_{\sigma \in [\sigma_{j-1},\sigma_j]} \left( \slope \calE t{\invcu(\sigma)}{-} 1 \right) 
\\
& 
\leq
\md{\invcu(\sigma_{j-1})}{\invcu(\sigma_j)}  +
 \sum_{l=1}^L  \md{\invcu(r_{l-1})}{\invcu(r_l)}\left( \slope \calE t{\invcu(r_l)} {-}1\right)
\\
& 
\stackrel{(1)}{\leq}
 \liminf_{k\to\infty}  \md{\invcur k(\sigma_{j-1}^k)}{\invcur k (\sigma_j^k)} +
 \sum_{l=1}^L  \liminf_{k\to\infty} \md{\invcur k (r_k(l{-}1))}{\invcur k (r_k(l))}
 \sqrt{2\mu_k \rstab {t_{{m}_k}^k}{\invcur k(r_k(l))}}
\\
&
\begin{aligned}
\stackrel{(2)}{\leq}
 \liminf_{k\to\infty}  \md{\invcur k(\sigma_{j-1}^k)}{\invcur k (\sigma_j^k)} &  +
  \liminf_{k\to\infty}   \sum_{l=1}^L 
\frac{\mu_k}2 \mdn^2(\invcur k (r_{k}(l{-}1)),\invcur k (r_k(l)))  \\ & +  \liminf_{k\to\infty}  \sum_{l=1}^L     \rstab {t_{{m}_k}^k}{\invcur k(r_k(l))}\,.
\end{aligned}
\end{aligned}
\end{equation}
For {\footnotesize (1)},  we have used that for every  $l=1,\ldots, L$
there exists a sequence $(r_k(l))_k$ fulfilling the aforementioned convergence property (2), and applied the forthcoming 
 Lemma \ref{l:AGS-2}  with the choice $\psi(r) = r +\frac12 r^2$ (cf.\ \eqref{psi-eps-intro}), so that $\psi^*(S) =\frac12 ((S{-}1)_+)^2$, 
   with  $\tau_k : = \mu_k^{-1}$, with $t_k: = t_{{m}_k}^k \to t$ as $k\to\infty$, and with
   $u_k: = \invcur k(r_k(l)) \to \invcu (r_l)$. 
   We then conclude that   (cf.\ \eqref{psi-MY} ahead for the definition of the generalized Moreau-Yosida approximation $\calY_{\mu_k^{-1}}^\psi (\calE)$) 
\begin{equation}
\label{partition-magic-2}
\begin{aligned}
\left(\slope \calE t{\invcu(r_l)}{-} 1 \right)=
 \left( \slope \calE t{\invcu(r_l)}{-}1 \right)_+
& \leq \liminf_{k\to\infty}\sqrt{2 \mu_k \left( \ene {t_{m_k}^k}{\invcur k (r_k(l))} -  \calY_{\mu_k^{-1}}^\psi (\calE)(t_{m_k}^k,\invcur k(r_k(l))) \right) }
\\&
= \liminf_{k\to\infty} \sqrt{2\mu_k \rstab t{\invcur k(r_k(l))}} \quad  \quad   \text{for all } l =1,\ldots, L. 
\end{aligned}
\end{equation}
Finally, for  {\footnotesize (2)} in  \eqref{case2.2} we have applied  
Young's inequality. 
\par
Observe that the term multiplied by $\mu_k$ featuring on the right-hand side of \eqref{case2.2} involves points that are extrema of holes in $C_k$. Therefore, it is estimated by $\Gap{\mdn}{\invcur k} {C_k} $,
whereas the third term is bounded by $\sum_{s\in S_{{m}_k}^k\setminus \{(S_{{m}_k}^k)^+\}} \rstab {t_{m_k}^k}{\invcur k(s)}$.
 Combining \eqref{ingredient-step-1},  
and \eqref{case2.2}, and summing over all the points of $(\sigma_i)_{i=1}^N \in \mathfrak{P}_f([0,C^+])$, we conclude the desired \eqref{we-need-this}. This finishes the proof  of Theorem \ref{th:lsc}. 
\QED
We conclude this section by giving the
\paragraph{\bf Outline of the proof of Proposition \ref{prop:added}.}
The argument borrows some ideas both from the proof of 
Theorem \ref{th:1}, and of Theorem \ref{th:2}.  Let us briefly sketch its steps.
\par
\noindent {\bf $\vartriangleright$ Compactness:}  We again apply Prop.\ \ref{prop:compactness} and  deduce the existence of a subsequence  
$(u_{n_k})_k$  converging to  some   $u\in \BV([0,T];\Xs)$ in the sense of   \eqref{convs} and  \eqref{props-lim};   hereafter 
	we will  again use the short-hands $u_k$,  $\mu_k$, and  $\vecostnamep {k}$ in place of 
	$u_{n_k}$, $\mu_{n_k}$, and 
	$\vecostnamep{\mu_{k}}$, respectively. 
	We will use the notation 
	\[
	\cmdn_{\mu_k}(u,v): = \md uv+ \tfrac{\mu_k}{2} \mdn^2(u,v), \qquad \qquad \cmdn_\mu (u,v): = \md uv+ \tfrac{\mu}{2} \mdn^2(u,v)\,,
	\]
and write $\mathrm{GapVar}_{\mdn}^{\mu_k}$, $\mathrm{GapVar}_{\mdn}^{\mu}$, $\mathcal{R}^{\mu_k}$, $\mathcal{R}^\mu$.
\par
\noindent {\bf $\vartriangleright$  The $\cmdn_\mu$-stability condition:} 	
	As in \emph{Claim $1$} within the proof of Thm.\ \ref{th:1}, we introduce the set
	$\widetilde{J} = \cap_{m\in \N}\cup_{k\geq m} \jump{u_k}$. First, we  prove that the limit curve $u$ fulfills the  stability condition 
	 ($\mathrm{S}_{\cmdn_\mu}$)
	at every $t\in [0,T] \setminus \widetilde J$ by passing to the limit  as $k\to\infty $ in the $\cmdn_{\mu_k}$-stability condition for the curves $u_k$, holding on $[0,T]\setminus \jump{u_k}$. Secondly, we deduce the validity of the  $\cmdn_\mu$-stability condition at every $t\in [0,T] \setminus \jump u$ by density argument, similarly as in the proof of Thm.\ \ref{th:1},  \emph{Claim $4$}. Here we exploit the closedness of the $\cmdn_\mu$-stable set $\mathscr{S}_{\cmdn_\mu}$, which is in turn ensured by the lower semicontinuity of $\mathcal{R}^\mu$. 
	 \par
\noindent {\bf $\vartriangleright$  The  upper energy estimate $\leq $  
in ($\mathrm{E}_{\mdn,\vecostname_\mu}$):} 	
We show 	
that
\begin{equation}
\label{uee-VE}
	\ene t{u(t)} + \Vari {\mdn,\vecostname_\mu}{u}0{t} \leq \ene 0{u(0)} +\int_0^t \pw s{u(s)} \dd s \quad \text{for all } t \in [0,T]
	\end{equation}
	by taking the $\liminf_{k\to\infty}$ in the analogous upper energy estimate for the curves $(u_k)_k$. Let us only comment on the proof of the key lower semicontinuity inequality
	 \begin{equation}
	 \label{key-lsc}
	\liminf_{k\to\infty}  \Vari {\mdn,\vecostname_{k}}{u_k}a{b} \geq \Vari {\mdn,\vecostname_\mu}{u}a{b} \qquad \text{for all } [a,b]\subset [0,T],
	 \end{equation}
since
for dealing with the other terms in \eqref{uee-VE} we repeat the 
 very same arguments as in the proofs of Thms.\ \ref{th:1} and \ref{th:2}.  
  \par
  First of all, we may suppose that the sequence $(u_k)_k$ complies with  the conditions \eqref{conditions-seqs}
 of Thm.\ \ref{th:lsc}.
 Along the footsteps of the proof of Thm.\  \ref{th:lsc},
 we introduce the Borel measures
 $\eta_k([a,b]): =  \Vari {\mdn,\vecostname_{k}}{u_k}a{b} $ and show that, up to a subsequence, they converge to a measure $\eta \geq   \mum u{\mathrm{d}}$. It then remains to deduce that
$
 \eta(\{t\}) \geq \vecost t{\lli u t }{\rli u t} \qquad \text{for all } t \in \jump u,
 $
 as well as  the analogue of \eqref{step3-bis},
 to conclude \eqref{key-lsc}. 
  With this aim we adapt the proof of Proposition \ref{prop:real-tech-diff} to show that
\[
 \liminf_{k\to\infty}  \Vari {\mdn,\vecostname_{k}}{u_k}{\alpha_k}{\beta_k} \geq 
 \vecost t{\lli u t }{\rli u t}
\]
 at every  point $t\in \jump u$, and for every pair of sequences $(\alpha_k)_k$, $(\beta_k)_k$ converging to $t$ and fulfilling 
 \eqref{conditions-diag}. 
 Hence, we reparameterize the curves $u_k$ in the very same way as in
  Step $1$ of the proof of Prop.\ 
 \ref{prop:real-tech-diff}.  By virtue of the a priori estimates  from Step $2$,
 the compactness arguments in Step $3$ yield the existence of a Lipschitz continuous limit  curve $\invcu: C \to \Xs$, with $C \Subset [0,\infty)$
 and $\invcu(C^-)= \lli u t$, $\invcu(C^+)= \rli u t$.
 Here, we can no longer replace  $C$ with the interval $[0,C^+]$ as in the proof of 
  Prop.\ 
 \ref{prop:real-tech-diff}, but we can still observe property 
 \eqref{che-fatica},
 based on \cite[Thm.\ 5.3]{SavMin16}. We now show that 
 \begin{equation}
 \label{what-we-want-added}
    \liminf_{k\to\infty} \Vari{\mdn,\vecostname_k}{u_{k}}{\alpha_k}{\beta_k}   \geq 
  \tcost{\VE}{t}{\invcu}{C}   \geq \vecost t{\lli u t }{\rli u t}\,.
  \end{equation}
The $\liminf$-inequality for the $\Varname{\mdn}$ contribution to $\Varname{\mdn,\vecostname_k}$ easily follows from the aforementioned compactness arguments.
For the $\mathrm{GapVar}_{\mdn}^{\mu_k}$-contribution (which depends on the parameter $\mu_k$ via the viscous correction $\tfrac{\mu_k}{2} \mdn^2$), it is essential to use property  \eqref{che-fatica}. 
For the $\mathcal{R}^{\mu_k}$ contribution, we can adapt the arguments from the discussion of Case (2) in Step $5$ of the proof of Prop.\  \ref{prop:real-tech-diff}, also exploiting 
the $\liminf$-estimate 
\[
(t_k\to t, \ x_k \to x ) \ \Rightarrow \ \ \liminf_{k\to\infty} \mathcal{R}^{\mu_k}(t_k,x_k)\geq \mathcal{R}^{\mu}(t,x).
\]
This concludes the proof of \eqref{key-lsc}.

  \noindent {\bf $\vartriangleright$  The  lower energy estimate $\geq $  
in   ($\mathrm{E}_{\mdn,\vecostname_\mu}$):} 	 
It follows from \cite[Thm.\ 6.5]{SavMin16}. Again, the energy convergence $\ene t{u_k(t)} \to \ene t{u(t)}$ for every $t\in [0,T]$ follows from the limit passage in the energy balance.
\QED
  \appendix 
  \section{Auxiliary results}
  \label{s:appendix}
\noindent  We start by fixing the representation formula \eqref{formula-integrale} for the transition cost $\tcost{\BV}{t}{\invcu}{[0,C^+]}$. In the upcoming statement, we replace the functional $u \mapsto \slope \calE t{u} \vee 1 $ by a general 
  \[
  g: \Xs \to \R \quad \text{ positive and lower semicontinuous. }
  \]
\begin{lemma}
\label{l:absolute-continuity}
Let $v \in \AC([a,b];\Xs)$. Then, there holds
\begin{equation}
\label{int-partition}
\int_a^b |v'|(s) g(v(s)) \dd s  =\sup\left\{ \sum_{i=1}^N \md{v(\sigma_{i-1})}{v(\sigma_i)} \inf_{\sigma \in [\sigma_{i-1},\sigma_i]} g(v(\sigma))  \, : \ (\sigma_i)_{i=1}^N \in \mathfrak{P}_f([a,b]) \right\}=: \mathsf{S}\,.
\end{equation}
 In particular, the map $s \mapsto |v'|(s) g(v(s))$ is integrable on $[a,b]$ if and only if $
\mathsf{S}<\infty.$
\end{lemma}
\begin{proof}
Let us fix $(\sigma_i)_{i=1}^N \in \mathfrak{P}_f([a,b]) $.  Observe that
\[
 \md{v(\sigma_{i-1})}{v(\sigma_i)} \inf_{\tilde\sigma \in [\sigma_{i-1},\sigma_i]} g(v(\tilde\sigma))  \stackrel{(1)}{\leq} \int_{\sigma_{i-1}}^{\sigma_i} |v'|(\sigma)  \inf_{\tilde\sigma \in [\sigma_{i-1},\sigma_i]} g(v(\tilde\sigma))  \dd \sigma
 \leq  \int_{\sigma_{i-1}}^{\sigma_i} |v'|(\sigma)   g(v(\sigma))  \dd \sigma
\]
with {\footnotesize (1)} due to  \eqref{metric_dev}. 
Therefore, upon summing up over the index $i=1,\ldots, N$ and using that $(\sigma_i)_{i=1}^N $ is arbitrary, we conclude 
\[
\int_a^b |v'|(s) g(v(s)) \dd s \geq \mathsf{S}\,.
\]
\par
As for the converse inequality, we now consider a partition $a=\sigma_1<\ldots<\sigma_i <\ldots = \sigma_N=b$ with fineness $\tau: = \max_{i=1,\ldots,N} (\sigma_i - \sigma_{i-1})$ and introduce the functions
\[
\underline{\sigma}_\tau,\, \overline{\sigma}_\tau : [a,b] \to [a,b] \quad  \text{ defined by } \quad \begin{cases}
\overline{\sigma}_\tau(s): = \sigma_i \quad \text{if } s \in (\sigma_{i-1},\sigma_{i}],
\\
\underline{\sigma}_\tau(s): = \sigma_{i-1} \quad \text{if } s \in [\sigma_{i-1},\sigma_{i}),
\end{cases} 
\]
with $\underline{\sigma}_\tau(b): = b$ and $\overline{\sigma}_\tau(a): = a$. 
Taking into account the definition \eqref{m-derivative} of the metric derivative $|v'|$, it is a standard matter to check that, on the one hand,
\begin{equation}
\label{point-wise-convergence}
\lim_{\tau \down 0}
\frac1{(\overline{\sigma}_\tau(s){-} \underline{\sigma}_\tau(s))}  \md{v(\underline{\sigma}_\tau(s))}{v(\overline{\sigma}_\tau(s))} \to |v'|(s) \qquad \text{for almost all } s \in (a,b)\,. 
\end{equation}
On the other hand, 
exploiting the lower semicontinuity of $g$, 
we observe that for every $s\in [a,b]$ there exists $\sigma_{\mathrm{min},\tau}(s) \in [\underline{\sigma}_\tau(s),\overline{\sigma}_\tau(s)]$ such that
\[
\inf_{\sigma \in [\underline{\sigma}_\tau(s),\overline{\sigma}_\tau(s)]} g(v(\sigma)) = g (v(\sigma_{\mathrm{min},\tau}(s))).
\]
Since $\sigma_{\mathrm{min},\tau}(s) \to s $ as $\tau\down 0$, by the continuity of $v$ and the lower semicontinuity of $g$ we then have 
\[
\liminf_{\tau\down 0}  g (v(\sigma_{\mathrm{min},\tau}(s))) \geq g(v(s)) \qquad \text{for all } s\in [a,b].
\]
Therefore, by the Fatou Lemma  we have
\[
\begin{aligned}
\mathsf{S} & \geq
\liminf_{\tau\down0} \sum_{i=1}^N \md{v(\sigma_{i-1})}{v(\sigma_i)} \inf_{\sigma \in [\sigma_{i-1},\sigma_i]} g(v(\sigma))  \\
 & = \liminf_{\tau\down0}  \int_a^b \frac1{(\overline{\sigma}_\tau(s){-} \underline{\sigma}_\tau(s))}  \md{v(\underline{\sigma}_\tau(s)}{v(\overline{\sigma}_\tau(s))}\,   g (v(\sigma_{\mathrm{min},\tau}(s))) \dd s &  \geq \int_a^b |v'|(s) g(v(s)) \dd s\,.
 \end{aligned}
\]
and we then conclude \eqref{int-partition}.
\end{proof}
\par
We conclude this Appendix by  extending the \emph{duality formula} 
from \cite[Lemma 3.1.5]{AGS08} for the  (squared) metric  slope  $|\mathrm{D} \calE|^2(t,\cdot)$,  $t\in [0,T]$ fixed, namely
\begin{equation}
\label{duality-quadratic}
\begin{gathered}
\frac12 |\mathrm{D} \calE|^2(t,u)= \limsup_{\tau \down 0} \frac{\calE(t,u)-\calE_\tau(t,u)}{\tau} \quad \text{with } 
\\
\calE_\tau(t,u): = \inf_{v\in \Xs} \left\{ \frac1{2\tau} \mdn^2(u,v) + \calE(t,v) \right\} \quad \text{the Moreau-Yosida approximation of } \calE(t,\cdot)
\end{gathered}
\end{equation}
(with slight abuse of notation). 
We consider
 the case in which  the  dissipation potential underlying the definition of Moreau-Yosida approximation  is no longer  the quadratic
 $\psi(r): = \frac12 r^2$,
  but a general
  function
  \begin{equation}
\label{superlinear-case}
\psi: [0,\infty)\to [0,\infty) \text{ convex, l.s.c., with } \psi(0)=0 \text{ and }   \lim_{r\up\infty}\frac{\psi(r)}r =\infty. 
\end{equation}
With $\psi$ we may associate  the \emph{generalized Moreau-Yosida} approximation 
of the  functional  $\calE(t,\cdot): \Xs \to \R$, via the formula (again, with slight abuse of notation, we write $ \calY_\tau^\psi(\calE)(t,u)$ in place of $ \calY_\tau^\psi(\calE(t,\cdot))(u)$) 
\begin{equation}
\label{psi-MY}
 \calY_\tau^\psi(\calE)(t,u): = \inf_{v\in \Xs} \left(  \tau \psi \left( \frac{\md{u}v}\tau \right) +\calE(t,v)  \right) \qquad \text{for } (t,u)\in [0,T]\times \Xs,\ \tau>0. 
\end{equation}
Combining the coercivity condition \eqref{Ezero} with the superlinear growth of $\psi$, it is straightforward to check that 
\[
M_\tau^\psi(\calE)(t,u):=  \mathrm{Argmin}_{v\in \Xs} \left( \tau \psi \left( \frac{\md{u}v}\tau \right)+ \calE(t,v)  \right) \neq \emptyset \qquad  \text{for all  } (t,u)\in [0,T]\times \Xs,\ \tau>0. 
\]
We have the following counterpart to 
 \cite[Lemma 3.1.5]{AGS08}. 
\begin{lemma}
\label{l:AGS}
There holds
\begin{equation}
\label{ext-duality}
 \psi^* \left(|\mathrm{D} \calE|(t,u) \right) = \limsup_{\tau\to 0} \frac{\calE(t,u) - \calY_\tau^\psi(\calE)(t,u)} \tau \qquad \text{for all } (t,u) \in [0,T]\times\Xs. 
\end{equation}
\end{lemma}
The \emph{proof} follows by trivially adapting the argument for 
  \cite[Lemma 3.1.5]{AGS08}.
 We conclude this Appendix with   
the following lower semicontinuity result, which is crucially used in the proof of Proposition \ref{prop:real-tech-diff}. 
\begin{lemma}
\label{l:AGS-2}
Assume \eqref{Ezero}, \eqref{uscPower-bis}, and \eqref{superlinear-case}. 
Let $(\tau_k)_k \subset (0,\infty)$,  $(t_k)_k \subset [0,T]$,  and $(u_k)_k  \subset \Xs$ fulfill   $\tau_k\down 0$, $t_k \to t$, and $u_k\to u$ for some 
$(t,u) \in [0,T]\times \Xs$,  with $\sup_{k\in \N}  \calE(t_k,u_k) \leq C$. Then,
\begin{equation}
\label{lsc-formula}
\liminf_{k\to\infty} \frac{\calE(t_k,u_k) - \calY_{\tau_k}^\psi(\calE)(t_k,u_k)} {\tau_k}\geq \psi^* \left(|\mathrm{D} \calE|(t,u) \right). 
\end{equation}
\end{lemma}
\begin{proof}
For every $k\in\N$, let  $u_{\tau_k}^k \in M_{\tau_k}^\psi(\calE)(t_k,u_k)$. We have that 
\[
\begin{aligned}
 \frac{\calE(t,u_k) - \calY_{\tau_k}^\psi(\calE)(t_k,u_k)} {\tau_k} & =  \frac{\calE(t_k,u_k) -  \calE(t_k,u_{\tau_k}^k) - \tau_k \psi \left( \frac{\displaystyle \md{u_k}{u_{\tau_k}^k}}{\displaystyle \tau_k} \right)  } {\tau_k}
 \\
  &  \geq  \frac1{\tau_k} \int_0^{\tau_k} \psi^*  \left(|\mathrm{D} \calE|(t_k,u_r^k) \right) \dd r, 
\end{aligned}
\]
where the latter  estimate  follows from \cite[Lemma 4.5]{RMS08}, with
 $u_r^k$ is a (measurable) selection in $M_{r}^\psi(\calE)(t_k,u_k)$ for  $r\in (0,\tau_k)$. Observe that  
 $\liminf_{k\to\infty}  \psi^*  \left(|\mathrm{D} \calE|(t_k,u_r^k) \right) \geq  \psi^*  \left(|\mathrm{D} \calE|(t,u) \right) $ taking into account that 
 $u_r^k\to u$ as $k\to\infty$  for every $r\in (0,\tau_k)$, cf.\ the proof of 
  \cite[Lemma 4.5]{RMS08},
  and using the lower semicontinuity of $|\mathrm{D}\calE|$ granted by \eqref{uscPower-bis}.  Then,  by Fatou's lemma we have 
  \[
  \liminf_{k\to\infty}  \frac1{\tau_k} \int_0^{\tau_k} \psi^*  \left(|\mathrm{D} \calE|(t_k,u_r^k) \right) \dd r \geq  \psi^* \left(|\mathrm{D} \calE|(t,u) \right),
  \]
  which concludes the proof of \eqref{lsc-formula}. 
\end{proof}


{\small 
\bibliographystyle{alpha}
\bibliography{ricky_lit}
}

\end{document}